\documentclass[12pt,a4paper]{article}%
\usepackage{amssymb}
\usepackage{amsmath}
\usepackage{amsfonts}
\usepackage[dvips]{graphicx}
\usepackage{mathrsfs}
\usepackage[pdftex]{hyperref}
\usepackage{pdfpages}%
\usepackage[all]{xy}
\providecommand{\U}[1]{\protect\rule{.1in}{.1in}}
\hypersetup{
	colorlinks=true,
	citecolor = blue
}
\newtheorem{theorem}{Theorem}

\newtheorem{corollary}[theorem]{Corollary}

\newtheorem{definition}[theorem]{Definition}
\newtheorem{assumption}[theorem]{Assumption}
\newtheorem{example}[theorem]{Example}

\newtheorem{lemma}[theorem]{Lemma}

\newtheorem{proposition}[theorem]{Proposition}
\newtheorem{remark}[theorem]{Remark}

\newenvironment{proof}[1][Proof]{\noindent\textbf{#1.} }{\ \rule{0.5em}{0.5em}}

\def\pot{^}
\def\menor{<}
\def\mayor{>}
\def\Sd{\mathbb{S}^{d-1}}
\def\rd{\mathbb{R}^d}

\def\Bd{\mathbb{B}^d}
\begin{document}
	
	\author{V\'{\i}ctor A. Vicente-Ben\'{\i}{}tez\\{\small Instituto de Matemáticas de la U.N.A.M. Campus Juriquilla}\\{\small Boulevard Juriquilla 3001, Juriquilla, Querétaro C.P. 076230 M\'{e}xico }\\ {\small  va.vicentebenitez@im.unam.mx } }
	\title{Generalized Poisson kernel and solution of the Dirichlet problem for the radial Schr\"odinger equation}
	\date{}
	\maketitle
	
	\begin{abstract}
		We present an explicit construction of the solution to the Dirichlet boundary value problem for the radial Schr\"odinger equation in the unit ball, with a complex-valued potential $V$ satisfying the condition $\int_0^1r|V(r)|dr<\infty$. The solution is based on the construction of an explicit orthogonal set of solutions for the radial equation. In the case of a Dirichlet problem with boundary data in $W^{\frac{1}{2},2}(\Sd)$, the solution is expressed as a series expansion in terms of the so-called formal spherical polynomials. We establish conditions for the solvability and uniqueness of the Dirichlet problem.  Based on this series representation, we introduce the concept of generalized Poisson kernel, develop its main properties, and investigate the conditions under which the Dirichlet problem, with a boundary condition being a complex Radon measure on $\mathbb{S}^{d-1}$, admits a solution in the sense of a distributional boundary values.
	\end{abstract}
	
	\newline 
	
	
	\section{Introduction}
	
	In the present work, we study the radial Schr\"{o}dinger equation 
	\begin{equation}\label{eq:radialintro}
		\left( \Delta-V(|x|)\right) u(x)=0,  \quad x\in \Bd,
	\end{equation}
	where $\Delta=\sum_{i=1}^{d}\frac{\partial^{2}}{\partial
		x_{i}^{2}}$ is the $d$-dimensional Laplacian, $\Bd$ denotes the unit ball in $\mathbb{R}^d$ with $d\geq 2$, and the potential $V$ is a complex-valued function
	depending only on the radial component $r=|x|$. We assume that $V$ satisfies the condition 
	\begin{equation}\label{eq:potentialconditionintro}
		\int_0^1r|V(r)|dr<\infty.
	\end{equation}
	The aim of this work is to provide an explicit construction of the solution to the Dirichlet problem associated with Eq. \eqref{eq:radialintro} in $\Bd$. We consider two formulations of the Dirichlet problem. The first is the standard formulation with {\it trace} boundary data $\varphi\in W^{\frac{1}{2},2}(\Bd)$, which consists of finding a weak solution $u\in W^{1,2}(\Bd)$ of \eqref{eq:radialintro} such that $\operatorname{tr}_{\Sd}u=\varphi$. The second formulation consists of finding a distributional solution $u\in L^1_{loc}(\Bd)$ of \eqref{eq:radialintro}, given a boundary data $\varphi$ interpreted as a distribution over $C^{\infty}(\Sd)$ (for instance, $\varphi\in L^2(\Sd)$). The solution $u$ is required to satisfy that its radial traces $u(r\cdot)$ converge to $\varphi$ in the weak-* topology of distributions as $r\rightarrow 1^{-1}$. This approach has taken on important relevance in the study of boundary value problems for elliptic equations and systems and in the study of their corresponding generalized Hardy spaces \cite{carmichael,estrada,straude,wblair1}. For example, in \cite{estrada}, it is shown how the Dirichlet problem with distributional boundary data for harmonic functions can be uniquely solved using the standard Poisson kernel. A problem of considerable interest is the explicit construction of the Poisson kernel for Schrodinger-type equations. In \cite{Quiao2}, a generalized Poisson kernel was explicitly constructed for Laplace and Helmholtz-type equations in cylindrical domains. This construction is based on the explicit computation of certain eigenfunctions of the Laplace operator. In  \cite{Quiao1}, the same author proposed a construction of a generalized Poisson integral representation for the radial Schr\"odinger equation in a cylindrical domain, under the assumption that the potential $V$ is positive and satisfies the condition $\lim_{r\rightarrow\infty}r^2V(r)=\kappa\in [0,\infty)$, and for continuous Dirichlet data. In \cite{frazier}, the conditions under which positive solutions of a Schrödinger equation admit a representation via a generalized Poisson kernel are established. This kernel is constructed using a Neumann operator series involving the harmonic Poisson and Green kernels. In the case when $V$ is analytic in the radial component $r$, it is known that solutions of \eqref{eq:radialintro} can be represented by means of an integral operator acting on harmonic functions \cite{bergman,gilbert3}. This result was extended in \cite{gilbert1} to potentials $V\in C^1[0,1]$ . The corresponding integral operator is called a {\it transformation} or {\it transmutation operator} for the radial Schr\"odinger equation \eqref{eq:radialintro}. A solution to the Dirichlet problem expressed in terms of the inverse of the transmutation operator is provided in \cite{gilbert4}. In \cite{mineradial1}, it is shown that the transmutation operator transforms the harmonic Bergman space onto the space of $L^2$ classical solutions of \eqref{eq:radialintro}, and an orthogonal basis of solutions takes the form
	\begin{equation}\label{eq:formapolynomialsintro}
		U_m(x)=|x|^{\frac{1-d}{2}}y_m(|x|)p\left(\frac{x}{|x|}\right),
	\end{equation}
	where $p$ is a spherical harmonic of degree $m$ and $\{y_m\}_{m=0}^{\infty}$ are the solutions of the perturbed Bessel equations $-y''_m+\left(\frac{\ell_m(\ell_m+1)}{r^2}+V(r)\right)y_m=0$ in $(0,1)$, with $\ell_m=m+\frac{d-3}{2}$.  The existence of such a system relies entirely on the existence of the transmutation operator, while the completeness of the system depends intrinsically on the fact that the transmutation operator has a continuous inverse. However, for the case when $V\not\in C^1[0,1]$, there is no explicit result that guarantees the existence of such a natural transmutation operator or its analytical properties.
	
	The first goal of this work is to extend the construction of a system of the form \eqref{eq:formapolynomialsintro} to the case where $V$ is complex valued and satisfies \eqref{eq:potentialconditionintro}. A decay condition of this type ensures the existence of regular solutions of the perturbed Bessel equations \cite{neumann2}.  An explicit construction of the solutions by means of the {\it spectral parameter power series} is presented in \cite{sppsbessel}, for potentials $V\in C(0,1]$ satisfying the asymptotic condition $V(r)=O(r^{\alpha})$ as $r\rightarrow 0^+$, for some $\alpha>-2$. We present a generalization of the construction presented in \cite{mineradial1}, which only uses the fact that the potential $V$ satisfies \eqref{eq:potentialconditionintro}. Next, we construct an orthogonal system of solutions of the form \eqref{eq:formapolynomialsintro} that we call the {\it formal spherical polynomials}. For $d\geq 3$, the orthogonality is given with respect to the inner product of $W^{1,2}(\Bd)$. It is important to note that although explicit solutions can be obtained via the separation of variables for certain regular potentials, our approach provides a general solution for any complex-valued potential satisfying the condition \eqref{eq:potentialconditionintro}. In addition, we present a constructive algorithm and establish the conditions under which the resulting series converges.
	
	Our second aim is the explicit construction of the solution to the Dirichlet problem. Once the orthogonal system is obtained, we analyze the Dirichlet problem with boundary data $\varphi\in W^{\frac{1}{2},2}(\Bd)$, and show that the solution can be expressed as a Fourier-series in terms of the formal spherical polynomials, with convergence in $W^{1,2}(\Bd)$. We establish additional conditions on the potential $V$ to ensure the uniqueness of the solution, and we conclude that the formal spherical polynomials form an orthogonal basis for the space of weak solutions in $W^{1,2}(\Bd)$ of Eq. \eqref{eq:radialintro}.
	
	From the obtained series representation, we derive a formula for the generalized Poisson kernel and show that, when $\varphi$ is a distribution on $C(\Sd)$ (i.e., a Radon measure on $\Sd$), Eq. \eqref{eq:radialintro} admits a distributional solution $u\in L^2(\Bd)\cap C(\Bd)$ given in terms of the Poisson integral, whose distributional boundary value is precisely $\varphi$. For the case $d=2$, this result extends to boundary values that are distributions over $C^{\infty}(\mathbb{S}^1)$.
	
	The paper is structured as follows. Section 2 provides background on the main properties of spherical gradients, spherical harmonics, and Sobolev spaces on the sphere. Section 3 is devoted to the explicit construction of an orthogonal system of solutions for equation \eqref{eq:radialintro}, including a procedure for constructing solutions to the associated Bessel equations. In Section 4, we obtain the solution to the Dirichlet problem with trace boundary data and establish the conditions on the potential $V$ required for uniqueness. Section 5 presents the construction of the generalized Poisson kernel and establishes the solution to the distributional boundary value problem. Finally, Section 6 discusses the details of these constructions in the case $d=2$. Appendix A contains technical proofs of certain properties of the gradient and spherical harmonics.

	\section{Background on spherical harmonics and Sobolev spaces on the sphere}
	
	Let $d\in \mathbb{N}$ with $d\geq 2$.  The open ball and the sphere with radius $r>0$ centered at $x_0\in \rd$ are denoted by $B_r^d(0)$ and $S_r^d(0)$. As is usual, we denote the unit ball and sphere by $\Bd=B_1^d(0)$ and $\Sd=S_1^d(0)$.
	Given a topological linear space $X$, its topological dual is denoted by $X'$, and the action of an element $f\in X'$ on $x\in X$ is denoted by $(f|x)_X$. Throughout the text, we use the notation $\mathbb{N}_0=\mathbb{N}\cup\{0\}$.
	For $1\leq p\leq \infty$, the space of $L\pot p$ functions on the ball with respect to the Lebesgue measure is denoted by $L\pot p(\Bd)$, while the corresponding space in the sphere, with the Borel surface measure $d\sigma$, is denoted by $L\pot p(\Sd)$. The Sobolev space of functions in $L\pot p(\Bd)$ whose distributional partial derivatives up to order $k\in \mathbb{N}$ also belong to $L\pot p(\Bd)$ is denoted by $W\pot{k,p}(\Bd)$. The corresponding local spaces are denoted by $L_{loc}^p(\Bd)$ and $W^{k,p}_{loc}(\Bd)$. For $1\menor p\menor \infty$, the space $W\pot{1,p}_0(\Bd)$ is the closure of $C_0\pot{\infty}(\Bd)$ in $W\pot{1,p}(\Bd)$, and $W\pot{-1,p}(\Bd):=(W_0\pot{1,p}(\Bd))'$. The space of test functions $C_0\pot{\infty}(\Bd)$ is denoted by $\mathscr{D}(\Bd)$, while for the sphere, we set $\mathscr{D}(\Sd)=C\pot{\infty}(\Sd)$. The surface area of $\Sd$ is denoted by $\omega_{d-1}$.  
	
	Every point $\xi=(\xi_1,\dots,\xi_d)\in \Sd$ can be written in terms of the generalized spherical coordinates
	\begin{equation}\label{eq:sphericalcoor}
		\xi_i=\begin{cases}
			\prod_{j=1}\pot{d-1}\sin\theta_{j},& \text{ if  }\; i=1,\\
			\cos\theta_{i-1}\prod_{j=i}\pot{d-1}\sin\theta_j,& \text{ if   }\; 2\leq i\leq d,
		\end{cases}
	\end{equation}
	where $0\leq \theta_1\leq 2\pi$ and $0\leq \theta_j\leq \pi$, $j=2,\dots,d-1$ (here, we use the convention that $\prod_{\emptyset}=1$). We denote by $\widehat{\mathbf{r}}$ the unit vector that corresponds exactly to the spherical coordinates \eqref{eq:sphericalcoor} (that is, $\widehat{\mathbf{r}}$ is the unit normal vector on $\Sd$ at the point $\xi(\theta_1,\dots, \theta_{d-1})$). On the other hand, let $\vec{\boldsymbol{\theta}}_j=\left(\frac{\partial \xi_i}{\partial \theta_j}\right)_{i=1}\pot{d}$, and $\Theta_j:=|\vec{\boldsymbol{\theta}_j}|$, $j=1,\dots,d-1$. 
	
	The following lemma summarizes the properties of the spherical coordinates.
	
	\begin{lemma}\label{Lemma:sphericalgradient}
		If we denote $\widehat{\boldsymbol{\theta}_j}:=\frac{1}{\Theta_j}\vec{\boldsymbol{\theta}}_j$, $j=1,\dots, d-1$, then $\{\widehat{\mathbf{r}},\widehat{\boldsymbol{\theta}}_1,\dots,\widehat{\boldsymbol{\theta}}_{d-1}\}$ forms an orthonormal basis, and for $u\in W\pot{1,2}(\Bd)$, its gradient $\nabla u$ can be decomposed as
		\begin{equation}\label{eq:sphericalgrad1}
			\nabla u=\frac{\partial u}{\partial r}\widehat{\mathbf{r}}+\frac{1}{r}\nabla_{\Sd}u,
		\end{equation}
		where $\nabla_{\Sd}$ is called the {\it spherical gradient} given by
		\begin{equation}\label{eq:sphgrad2}
			\nabla_{\Sd}u= \sum_{j=1}\pot{d-1}\frac{1}{\Theta_j}\frac{\partial u}{\partial \theta_j}\widehat{\boldsymbol{\theta}}_j.
		\end{equation}
		
	\end{lemma}
	Note that $\nabla_{\Sd}$ acts only on the coordinates $\theta_1,\dots,\theta_{d-1}$, and is the orthogonal projection of $\nabla u$ onto the tangent space $\mathcal{T}_{\xi}(\Sd)$ at the point $\xi\in \Sd$.  This fact is well known for dimensions $d=2,3$. The existence of the spherical gradient in general dimensions is established in \cite{atkinsonsph}. However, as explicit computations in spherical coordinates are not provided there, we include them in Appendix \ref{Sec:Appendix} for the reader's convenience.
	
	A well-known result concerns the decomposition of the Laplace operator.  Let $\Delta:=\sum_{j=1}\pot{d}\frac{\partial}{\partial x_j}$ be the Laplacian in $\rd$. It is known that $\Delta$ can be written in the form 
	\begin{equation}\label{eq:sphelaplace}
		\Delta= \frac{1}{r\pot {d-1}}\frac{\partial}{\partial r}r\pot{d-1}\frac{\partial}{\partial r}+\frac{1}{r\pot 2}\Delta_{\Sd},
	\end{equation}
	where $\Delta_{\Sd}$ is the {\it spherical Lapacian} (also known as the {\it Laplace-Beltrami operator}, see \cite[Lemma 1.4.1]{day} or \cite[Sec. 3.1]{atkinsonsph}), acting only on the coordinates $\theta_1,\dots,\theta_{d-1}$. As is well known, $\int_{\Bd}f(x)dv_x=\int_0^1r^{d-1}\int_{\Sd}u(r\xi)d\sigma_{\xi}dr$ for $u\in L^1(\Bd)$.
	
	\begin{remark}
		Let $u\in C^1(\overline{\Bd})$. Since the exterior unit normal of $\Sd$ at the point $\xi(\theta_1,\dots,\theta_{d-1})$ is precisely $\widehat{\mathbf{n}}$, then the normal derivative $\frac{\partial u}{\partial \nu}$ is given by
		\begin{equation}\label{eq:normalder}
			\frac{\partial u}{\partial \nu}=\nabla u\cdot \widehat{\mathbf{n}}\big{|}_{\Sd}=\frac{\partial u}{\partial r}\bigg{|}_{r=1}.
		\end{equation}
	\end{remark}
	
	The set of all homogeneous harmonic polynomials of degree $m\in \mathbb{N}_0$ in $\rd$ is denoted by $\mathcal{H}_m(\rd)$. For the case $d=2$, $\mathcal{H}_m(\mathbb{R}^2)$ is generated by the complex monomials $\{z^m,\overline{z}^m\}$. For $d\geq 3$, $\mathcal{H}_m(\rd)$ is finite dimensional and $d_m=\dim \mathcal{H}_m(\rd)$ is given by
	\[
	d_m=\begin{cases}
		1, & \text{ if  } m=0,\\
		d, & \text{ if  } m=1,\\
		\binom{d+m-1}{d-1}-\binom{d+m-3}{d-1}, & \text{ if  } m\geq 2,
	\end{cases}
	\]
	(see \cite[Prop. 5.8]{axler}. The set of the spherical harmonics of degree $m$ is given by $\mathcal{H}_m(\Sd):=\{p|_{\Sd}\,|\, p\in \mathcal{H}_m(\rd)\}$. It holds that $\dim \mathcal{H}_m(\Sd)=d_m$. Fix an orthonormal basis $\{Y_j\pot{(m)}\}_{j=1}\pot{dm}$ for $\mathcal{H}_m(\Sd)$. The following lemma establishes the orthogonality of the spherical harmonics and their spherical gradients.
	
	\begin{lemma}\label{Prop:sphericalprop1}
		\begin{itemize}
			\item[(i)] If $m\neq n$, then $\mathcal{H}_m(\Sd)\perp_{L\pot 2(\Sd)} \mathcal{H}_m(\Sd)$.
			\item[(ii)] The spherical harmonics are the eigenfunctions of the spherical Laplacian, that is, if $m\in \mathbb{N}_0$ and $p\in \mathcal{H}_m(\Sd)$, then 
			\begin{equation}\label{eq:SphericalEigen}
				\Delta_{\Sd} p=-m(m+d-2)p.
			\end{equation}
			\item[(iii)] $L\pot 2(\Sd)=\bigoplus_{m=0}\pot{\infty}\mathcal{H}_m(\Sd)$.
			\item[(iv)] If $m\in \mathbb{N}$ and $p,q\in \mathcal{H}_m(\rd)$, then
			\begin{equation}\label{eq:relatriongradients1}
				\int_{\Sd}p(\xi)\overline{q(\xi)}d\sigma_{\xi}=\frac{1}{m(d+2m-2)}\int_{\Sd}\nabla p(\xi)\overline{\nabla q(\xi)}d\sigma_{\xi}.
			\end{equation}
			\item[(v)] If $p\in \mathcal{H}_m(\Sd)$ and $q\in \mathcal{H}_m(\Sd)$ with $n\neq m$, then 
			\[
			\int_{\Sd}\nabla_{\Sd}p(\xi)\overline{\nabla_{\Sd}q(\xi)}d\sigma_{\xi}=0.
			\]
			\item[(vi)] For $m\geq 1$, and $j,i\in \{1,\dots,d_m\}$, we get
			\begin{equation}\label{eq:ortogonalgrad1}
				\int_{\Sd}\nabla_{\Sd}Y_j\pot{(m)}(\xi)\overline{\nabla_{\Sd}Y_k\pot{(m)}(\xi)}d\sigma_{\xi}=0,\quad \text{if  } j\neq k,       
			\end{equation}
			and 
			\begin{equation}\label{eq:normgradspherical}
				\int_{\Sd}|\nabla_{\Sd}Y_j\pot{(m)}(\xi)|\pot 2d\sigma_{\xi}=m(m+d-2).
			\end{equation}
		\end{itemize}
	\end{lemma}
	
	The proof is given in Appendix \ref{Sec:Appendix}.
	
	Let $\varphi\in L\pot2(\Sd)$. The Fourier coefficients of $\varphi$ with respect to the orthonormal basis $\{\{Y_j\pot{(m)}\}_{j=1}\pot{d_m}\}_{m=0}\pot{\infty}$ are denoted by
	\[
	\widehat{\varphi}_{m,j}:=\int_{\Sd}\varphi(\xi)\overline{Y_j\pot{(m)}}(\xi)d\sigma_{\xi},\quad m\in \mathbb{N}_0,\;  j=1,\dots, d_m. 
	\]
	
	There exists the bounded trace operator $\operatorname{tr}_{\Sd}:W\pot{1,2}(\Bd)\rightarrow L\pot{2}(\Sd)$, which satisfies $\operatorname{tr}_{\Sd}u=u|_{\Sd}$ for all $u\in C(\overline{\Bd})$ and $\operatorname{tr}_{\Sd}u=0$ iff $u\in W_0\pot{1,2}(\Bd)$. The image of $W\pot{1,2}(\Bd)$ under the trace operator is denoted by $W\pot{\frac{1}{2},2}(\Sd)$ and is called the Sobolev space of order $\frac{1}{2}$ (see \cite[pp. 102-106]{mclean}). The following characterization of the space of traces is given in terms of the Fourier coefficients (see \cite[Sec. 3.8]{atkinsonsph}).
	
	\begin{proposition}
		The space $W\pot{\frac{1}{2},2}(\Sd)$ consists of the functions $\varphi\in L\pot{2}(\Sd)$ satisfying the condition
		\begin{equation}\label{eq:sobolevtrace}
			\sum_{m=0}\pot{\infty}\sum_{j=1}\pot{d_m}\sqrt{m(m+d-2)}|\widehat{\varphi}_{m,j}|\pot 2\menor \infty. 
		\end{equation}
		The space $W\pot{\frac{1}{2},2}(\Sd)$ is a Banach space with the norm
		\begin{equation*}
			\|\varphi\|_{W^{\frac{1}{2},2}(\Sd)}:=\left\{\|\varphi\|_{L^2(\Sd)}^2+\sum_{m=0}\pot{\infty}\sum_{j=1}\pot{d_m}\sqrt{m(m+d-2)}|\widehat{\varphi}_{m,j}|\pot 2\right\}^{\frac{1}{2}}.
		\end{equation*}
	\end{proposition}

	\section{An orthogonal system of solutions for the  Schr\"odinger equation}
	
	Given $\alpha>0$ and $1\leq p\menor \infty$, we denote by $L_{\alpha}\pot p(0,1)$ the $L\pot p$-space on $(0,1)$ with respect to the Borel measure $r\pot{\alpha}dr$. Note that $r^{\alpha}dr$ is a finite measure on $(0,1)$, hence we have the continuous embedding  $L^p_{\alpha}(0,1)\hookrightarrow L_{\alpha}^q(0,1)$, whereas $q\leq p$ \cite[Prop. 6. 12]{folland}. For simplicity, we denote $L_0\pot 1(0,1)=L\pot{1}(0,1)$, the standard $L\pot p$ space with the Lebesgue measure. 
	Let $V$ be a measurable complex-valued function in $L_1\pot 1(0,1)$, that is, satisfying the condition
	\begin{equation}\label{eq:conditionV}
		\int_0\pot 1r|V(r)|dr\menor \infty.
	\end{equation}
	Through the text, we assume that $V$ satisfies \eqref{eq:conditionV}.
	
	We construct a system of solutions for the radial Schr\"odinger equation with potential $V$
	\begin{equation}\label{eq:radialSchr}
		-\Delta u+V(r)u=0,\quad \text{ in  }\;\; \Bd.
	\end{equation}
	
	\begin{definition}
		A function $u\in L\pot 1_{loc}(\Bd)$ is called a {\bf distributional} solution of Eq. \eqref{eq:radialSchr} if
		\begin{equation}\label{eq:distrsol}
			\int_{\Bd}u(-\Delta \varphi+V\varphi)=0\quad \forall \varphi\in \mathscr{D}(\Bd).
		\end{equation}
		A function $u\in W\pot{1,2}(\Bd)$ is a {\bf weak} solution of Eq. \eqref{eq:radialSchr} if
		\begin{equation}\label{eq:weaksol}
			\int_{\Bd}(\nabla u\cdot\nabla \varphi+Vu\varphi)=0\quad \forall \varphi\in \mathscr{D}(\Bd).
		\end{equation}
		The spaces of distributional and weak solutions are denoted by $\operatorname{Sol}_V\pot{dist}(\Bd)$ and $\operatorname{Sol}_V\pot w(\Bd)$, respectively.
	\end{definition}
	Note that in both definitions it is required that $\int_{\Bd}Vu\varphi$  be defined. This condition is fulfilled when $u$ or $V$ belongs to $L_{loc}\pot{\infty}(\Bd)$, for the case of distributional solutions, or when $V\in L_{loc}^2(\Bd)$, for weak solutions.
	
	\begin{proposition}\label{Prop:closedsubspace}
		If $V\in L_{d-1}^2(0,1)$, then the subspace $\operatorname{Sol}_V\pot w(\Bd)$ is closed in $W\pot{1,2}(\Bd)$. Furthermore,  if $V\in L_{d-1}^r(0,1)$, with $r=\max\{2,\frac{d}{2}\}$, then $u\in \operatorname{Sol}_V\pot{w}(\Bd)$ iff
		\begin{equation}\label{eq:weaksol2ndform}
			\int_{B}(\nabla u\cdot \nabla v+Vuv)=0\quad \forall v\in W_0\pot{1,2}(\Bd).
		\end{equation}
	\end{proposition}
	\begin{proof}
		The condition on $V$ implies that $\int_{\Bd} |V|^2=\omega_{d-1}\int_0^1r^{d-1}|V(r)|^2<\infty$. Consider the bilinear form
		\[
		B[u,\varphi]:= \int_{\Bd}(\nabla u\cdot\nabla \varphi+Vu\varphi),\quad u\in W\pot{1,2}(\Bd), \; \varphi\in \mathscr{D}(\Bd).
		\]
		Note that $\int_{\Bd}|Vu\varphi|\leq \omega_{d-1}\|V\|_{L_1\pot 2(0,1)}\|u\|_{L\pot 2(\Bd)}\|\varphi\|_{L\pot{\infty}(\Bd)}$, so $B$ is well defined for all $u\in W\pot{1,2}(\Bd) \text{ and } \varphi\in \mathscr{D}(\Bd)$. Suppose that $\{u_n\}\subset \operatorname{Sol}_V\pot w(\Bd)$ converges to $u\in W\pot{1,2}(\Bd)$. Hence, given $\varphi\in \mathscr{D}(\Bd)$, we have that  $B[u_n,\varphi]=0$ for all $n\in \mathbb{N}$ and 
		\begin{align*}
			|B[u,\varphi]|=|B[u-u_n,\varphi]|\leq \left(\sqrt{\operatorname{Vol}(\Bd)}+\omega_{d-1}\|V\|_{L_1\pot 2(0,1)}\right)\|u-u_n\|_{W\pot {1,2}(\Bd)}\|\varphi\|_{W\pot{1,\infty}(\Bd)}\rightarrow 0,
		\end{align*}
		as $n\rightarrow \infty$. Therefore, $B[u,\varphi]=0$ for all $\varphi\in \mathscr{D}(\Bd)$ and $u\in \operatorname{Sol}_V\pot w(\Bd)$. 
		
		Now, suppose that $V\in L_{d-1}^r(0,1)$ with $r=\max\{2,\frac{d}{2}\}$ (which implies that $V\in L^r(\Bd)$). First, suppose that $d\geq 3$. When $u,v\in W\pot{1,2}(\Bd)$, by the Sobolev embedding theorems \cite[Cor. 9.14]{brezis}, $u,v\in  L^{2^*}(\Bd)$, where $2^*=\frac{2d}{d-2}$. Applying the generalized H\"older inequality with $1=\frac{2}{d}+\frac{2}{2^*}$ we get 
		\begin{align*}
			|B[u,v]| & \leq \int_{\Bd}|\nabla u\cdot \nabla v|+\int_{\Bd}|Vuv|\\
			& \leq \|u\|_{W^{1,2}(\Bd)}\|v\|_{W^{1,2}(\Bd)}+\|V\|_{L^{\frac{d}{2}}(\Bd)}\|u\|_{L^{2^*}(\Bd)}\|u\|_{L^{2^*}(\Bd)}\\
			& \leq (1+\tilde{C}^2\|V\|_{L^{\frac{d}{2}}(\Bd)})\|u\|_{W^{1,2}(\Bd)}\|v\|_{W^{1,2}(\Bd)},
		\end{align*}
		where $\tilde{C}$ is the norm of the embedding $W^{1,2}(\Bd)\hookrightarrow L^{2^*}(\Bd)$. Hence the bilinear form $B: W\pot{1,2}(\Bd)\times W\pot{1,2}(\Bd)\rightarrow \mathbb{C}$ is bounded. Since $\mathscr{D}(\Bd)$ is dense in $W_0\pot{1,2}(\Bd)$, the continuity of $B$ implies \eqref{eq:weaksol2ndform}. The proof for $d=2$ is essentially the same, but using the fact that $u,v\in L^q(\mathbb{B}^2)$ for $1\leq q<\infty$ \cite[Cor. 9. 14]{brezis}, and the generalized H\"older inequality with $V\in L^2(\mathbb{B}^2)$ and $u,v\in L^4(\mathbb{B}^2)$.
	\end{proof}
	
	In particular, the first hypothesis holds when $V\in L_1^2(0,1)$, because $\int_0\pot1 r\pot{d-1}|V(r)|\pot 2dr\leq \int_0\pot 1r|V(r)|\pot 2\menor \infty$ ( $r\pot{d-2}\leq 1$, since $d\geq 2$).

	Let $p\in \mathcal{H}_m(\Bd)$. Following \cite{mineradial1}, we seek a solution $U_m$ to Eq. \eqref{eq:radialSchr} of the form 
	\begin{equation}\label{eq:radialpower1}
		U_m(r\xi)=r\pot m\alpha_m(r)p(\xi),\quad 0<r<1, \xi\in \Sd.
	\end{equation}
	for some radial function $\alpha_m(r)$. Substituting this ansatz into Eq. \eqref{eq:radialSchr} and using the decomposition of the Laplacian \eqref{eq:sphelaplace} along with the spherical eigenvalue identity \eqref{eq:SphericalEigen}, we obtain:
	
	\begin{align*}
		0= &\left[ r^{m}\left( \alpha_m''-V(r)\alpha_m\right)
		+r^{m-1}(2m+d-1)\alpha_m'\right]p \\
		& +\left[ r^{m-2}\alpha_m\left( m(m-1)+m-m(m+d-2)+m(d-1)\right) \right]p\\
		= & \left( \alpha_m''-V(r)\alpha_m+\frac{2m+d-1}{r}\alpha_m'\right) r^mp.
	\end{align*}
	Hence $U_m$  satisfies Eq. \eqref{eq:radialSchr} if $\alpha_m(r)$ is a solution of 
	\[
	\alpha_m''-V(r)\alpha_m+\frac{2m+d-1}{r}\alpha_m'=0,\quad 0<r<1.
	\]
	Set $\ell_m:= m+\frac{d-3}{2}$. Writing $\alpha_m(r):= \frac{y_m(r)}{r^{\ell_m+1}}$, the last equality is reduced to the {\it perturbed Bessel equation}
	\begin{equation}\label{eq:Besselperturbed}
		-y_m''+\frac{\ell _{m}(\ell
			_{m}+1)}{r^{2}}y_{m}+V(r)y_{m}=0,\quad 0<r<1.
	\end{equation}
	In order to obtain a solution $U_m\in L_{loc}^1(\Bd)$, we establish condition $\alpha_m(0)=1$. To obtain this, it suffices that the solution $y_m$ satisfies the asymptotic conditions 
	\begin{equation}\label{eq:asymptym}
		y_{m}(r)\sim r^{\ell _{m}+1}=r^{m+\frac{d-1}{2}},\;y_{m}^{\prime }(r)\sim
		(\ell _{m}+1)r^{\ell _{m}},\;r\rightarrow 0^{+}.
	\end{equation}
	These conditions characterize the so-called regular solution of the perturbed Bessel equation (see \cite{kostenko, serier}). Note that the critical case arises when $m=0$, $d=2$, because the asymptotic takes the form $y_0(r)\sim \sqrt{r}$, $y'_0(r)\sim \frac{1}{2\sqrt{r}}$, $r\rightarrow 0^+$.  
	\newline
	
	In what follows, {\it we assume that $d\geq 3$} (the case $d=2$ will be treated separately in Section \ref{Sec:d=2}).
	\newline

	The construction of the solutions $\{y_m\}_{m=0}^{\infty}$ was proposed in \cite{mineradial1} for the case $V\in C^1[0,1]$, and is based on the {\it spectral parameter power series} \cite{spps}. We now extend this construction to the more  general case when $\ell\mayor 0$, i.e., to the construction of the solution $w_{\ell}$ of the perturbed Bessel equation
	\begin{equation}\label{eq:perturbedbesselgral}
		-w''_{\ell}+\frac{\ell(\ell+1)}{r\pot 2}w_{\ell}+V(r)w_{\ell}=0,\quad 0\menor r\menor 1,
	\end{equation}
	satisfying the asymptotics
	\begin{equation}\label{eq:asymptgeneral}
		w_{\ell}(r)\sim r\pot{\ell+1},\quad w_{\ell}'(r)\sim (\ell+1)r\pot{\ell},\quad r\rightarrow 0\pot +.
	\end{equation}
	In particular, for each $m\in \mathbb{N}_0$ we have $y_m=w_{\ell_m}$ with $\ell_m=m+\frac{d-3}{2}$. The solution $w_{\ell}$ is constructed as a functional series of the form
	
	\begin{equation}\label{eq:sppsseries}
		w_{\ell}(r)=\sum_{k=0}^{\infty}\psi_k^{\ell}(r),
	\end{equation}
	where
	\begin{equation}\label{formalpowersbessel}
		\psi _{k}^{\ell}(r):=
		\begin{cases}
			r^{\ell+1}, & \mbox{ for }\;k=0, \\ 
			\displaystyle\int_{0}^{r}\mathcal{L}_{\ell}(r,s)V(s)\psi _{k-1}^{\ell}(s)ds, & 
			\mbox{ for }\;k\geq 1,
		\end{cases}
	\end{equation}
	and $\mathcal{L}_{\ell}(r,s)$ is the kernel
	\begin{equation}\label{greenfunction1}
		\mathcal{L}_{\ell}(r,s):=\frac{1}{2\ell +1}\left( \frac{r^{\ell+1}}{%
			s^{\ell}}-\frac{s^{\ell+1}}{r^{\ell}}\right) \quad 
		\mbox{ for
		}\;(r,s)\in (0,1]\times (0,1].  
	\end{equation}
	
	As shown in \cite[Sec. 6]{mineradial1}, the kernel $\mathcal{L}_{\ell}(r,s)$ satisfies the following properties: $\mathcal{L}_{\ell}\in C^2((0,1]\times (0,1])$; $\frac{\partial \mathcal{L}_{\ell}(r,s)}{\partial r}\big{|}_{r=s}=1$; $\frac{\partial^{2}}{\partial r^{2}}\mathcal{L}_{\ell}(r,s)=\frac{\ell
		(\ell+1)}{r^{2}}\mathcal{L}_{\ell}(r,s)$; and the estimates 
	\begin{equation}\label{eq:estimatesLm}
		|\mathcal{L}_{\ell}(r,s)|\leqslant\frac{2}{2\ell+1}\frac {
			r^{\ell+1}}{s^{\ell}}\quad\text{and }\;\;  \left| \frac{\partial}{\partial r}\mathcal{L}_{\ell}(r,s)\right|
		\leqslant\left( \frac{r}{s}\right) ^{\ell}\quad \text{for  }\; 0<s\leq r\leq 1.
	\end{equation} 
	Consequently, the functions $\psi_k^{\ell}$ satisfy the recursive relations
	\begin{equation}\label{eq:recursivepsi}
		(\psi_k^{\ell})'(r)=\int_0^r\frac{\partial \mathcal{L}_{\ell}(r,s)}{\partial s}V(s)\psi_{k-1}^{\ell}(s)ds,\quad (\psi_k^{\ell})''(r)=-V(r)\psi_{k-1}^{\ell}+\frac{\ell(\ell+1)}{r^2}\psi_k^{\ell}.
	\end{equation}
	\begin{proposition}\label{Prop:Solbessel}
		The series \eqref{eq:sppsseries}, as well as the series of their first derivative, converge absolutely and uniformly on $[0,1]$. In particular,  $w_{\ell}\in C^1[0,1]$. Furthermore,  $w_{\ell}'\in AC[\delta,1]$ for all $0<\delta<1$. The function $w_{\ell}$ satisfies the perturbed Bessel equation \eqref{eq:perturbedbesselgral} a.e. on $(0,1)$, and the asymptotic conditions \eqref{eq:asymptgeneral}.
	\end{proposition}
	\begin{proof}
		The proof is essentially the same as that of \cite[Th. 43]{mineradial1}, where the hypotheses $V\in C^1[0,1]$ was not used, but only the fact that $V\in L_1^1(0,1)$. We include the main details for the reader's convenience. 
		
		Using \eqref{eq:estimatesLm}, an induction argument yields the estimates
		\begin{align}
			|\psi_k^{\ell}(r)| & \leq \left(\frac{2}{2\ell+1}\right)^k\frac{r^{\ell+1}}{k!}\left(\int_0^rs|V(s)|ds\right)^k \quad \forall k\in \mathbb{N}_0,\\ \label{eq:bounpsik}
			|(\psi_k^{\ell})'(r)| &\leq \left(\frac{2}{2\ell+1}\right)^{k-1}\frac{r^{\ell}}{k!}\left(\int_0^rs|V(s)|ds\right)^k \quad \forall k\in \mathbb{N}.\\ \label{eq:bounpsikprime}
		\end{align}
		Thus,
		\begin{align}
			\sum_{k=0}^{\infty}|\psi_k^{\ell}(r)|&\leq r^{\ell+1}e^{\frac{2}{2\ell+1}\|V\|_{L_1^1(0,1)}},\label{eq:estimateseries1}\\
			\sum_{k=1}^{\infty}|(\psi_k^{\ell})'(r)|&\leq \left(\ell+\frac{1}{2}\right)r^{\ell}\left[e^{\frac{2}{2\ell+1}\|V\|_{L_1^1(0,1)}}-1 \right].\label{eq:estimateseries2}
		\end{align}
		The Weierstrass M-test ensures the absolute and uniform convergence of the series on $[0,1]$. From \eqref{eq:recursivepsi}, $\psi_k^{\ell}\in C^1[0,1]$ and hence $w_{\ell}\in C^1[0,1]$. Again, by \eqref{eq:recursivepsi}, $(\psi_k^{\ell})''\in L_1^1(\delta,1)$ for all $0<\delta<1$, and due to the estimates \eqref{eq:bounpsik} and \eqref{eq:bounpsikprime}, the series $\sum_{k=0}^{\infty}(\psi_k^{\ell})''$ converges in $L_1^1(\delta,1)$, and is not difficult to see that $w_{\ell}\in AC[\delta,1]$ for all $0<\delta<1$. The recursive relations $\eqref{eq:recursivepsi}$ imply that $w_{\ell}$ satisfies \eqref{eq:perturbedbesselgral} a.e. in $(0,1)$. Finally, in order to establish the asymptotics \eqref{eq:asymptgeneral}, we observe that 
		\begin{equation*}
			|w_{\ell}(r)-r^{\ell +1}|\leqslant \sum_{k=1}^{\infty }\left( \frac{2}{%
				2\ell+1}\right) \frac{r^{\ell +1}}{k!}\left(
			\int_{0}^{r}s|V(s)|ds\right) ^{k}=r^{\ell+1}\left\{ e^{ \frac{2%
				}{2\ell +1}\int_{0}^{r}s|V(s)|ds} -1\right\} .
		\end{equation*}%
		Thus, $\displaystyle\left\vert \frac{w_{\ell}(r)}{r^{\ell +1}}-1\right\vert
		\leqslant \left\{ e^{ \frac{2}{2\ell+1}\int_{0}^{r}s|V(s)|ds} -1\right\} $. The right-hand side tends to zero when $r\rightarrow
		0^{+}$. Therefore, $\lim_{r\rightarrow 0^{+}}\frac{w_{\ell}(r)}{r^{\ell+1}}=1$. The proof of the second asymptotic is similar.
	\end{proof}

	The following estimates for the functions $\alpha_m$ will be needed.
	
	\begin{lemma}\label{Lemma:boundsalpha}
		There exists constants $C_j>0$, $j=0,1,2$, such that the following estimates hold for the functions $\{\alpha_m\}_{m=0}\pot{\infty}$:
		\begin{itemize} 
			\item[(i)] $|\alpha_m(r)|\leq C_0$ for all $r\in [0,1]$, $m\in \mathbb{N}_0$.
			\item[(ii)] $|\alpha_m'(r)|\leq \frac{mC_1}{r}$ for all $r\in (0,1]$, $m\in \mathbb{N}$.
			\item[(iii)] $|\alpha_0'(r)|\leq \frac{C_2}{r}$ for all $r\in (0,1]$.
		\end{itemize}
	\end{lemma}
	\begin{proof}
		\begin{itemize}
			\item[(i)] From \eqref{eq:estimateseries1}, 
			\[
			|\alpha_m(r)|=\frac{1}{r^{\ell_m+1}}|y_m(r)|\leq e^{\frac{2}{2\ell_m+1}\|V\|_{L_1^1(0,1)}}.
			\]
			Since $2\ell_m+1=2m+d-2$,  it follows that $e^{\frac{2}{2\ell_m+1}\|V\|_{L_1^1(0,1)}}\leq e^{2\|V\|_{L^1(0,1)}}$ for all $m\in \mathbb{N}_0$. Then $C_0=e^{2\|V\|_{L^1(0,1)}}$ satisfies the condition in (i).
			\item[(ii)] Note that
			\[
			\alpha_m'(r)=\frac{y_m'(r)}{r^{\ell_m+1}}-(\ell_m+1)\frac{y_m(r)}{r^{\ell_m+1}}= \frac{1}{r}\left(\frac{y_m'(r)}{r^{\ell_m}}-(\ell_m+1)\alpha_m(r)\right).
			\] 
			By the estimate \eqref{eq:estimateseries2},
			\begin{align*}
				|y_m'(r)|& \leq (\ell_m+1)r^{\ell_m}+\left(\ell_m+\frac{1}{2}\right)r^{\ell_m}\left[e^{\frac{2}{2\ell_m+1}\|V\|_{L_1^1(0,1)}}-1\right] \leq 2r^{\ell_m}(\ell_m+1).
			\end{align*}
			Since $\lim\limits_{m\rightarrow\infty}\frac{\ell_m+1}{m}=1$, we can choose $M>0$ and $n_0\in \mathbb{N}$ such that $\ell_m+1\leq Mm$ for $m\geq n_0$. Taking $C_1=2\max\{\ell_1+1,\dots,\frac{\ell_{m_0-1}+1}{m_{0-1}},M\}(1+C_0)$, we obtain (ii).
			\item[(iii)] As in the previous point, $|\alpha_0'(r)|\leq \frac{1}{r}\left(\frac{y_0'(r)}{r^{\frac{d-3}{2}}}+\frac{d-1}{2}| \alpha_0(r)| \right)$, then $C_2=\frac{d-1}{2}(2+C_0)$.
		\end{itemize}
		
	\end{proof}

	\begin{theorem}\label{Th:Sol1}
		Given $p\in \mathcal{H}_m(\Bd)$, $m\in \mathbb{N}_0$, the function $U_m(x)=r^m\phi_m(r)p(\xi)$ belongs to $W^{1,2}(\Bd)\cap W^{2,1}(A_{\varepsilon}(0))$, where $A_{\varepsilon}(0)=\Bd\setminus \overline{B}^d_{\varepsilon}(0)$, and satisfies the Schr\"odinger equation a.e. in $A_{\varepsilon}(0)$ for all $0<\varepsilon<1$. Therefore, $U_m\in \operatorname{Sol}_V^w(\Bd)$.
	\end{theorem}
	\begin{proof}
		By Lemma \ref{Lemma:boundsalpha}(i), $U_m\in L^2(\Bd)$. To estimate the gradient, we use the decomposition \eqref{eq:sphericalgrad1} together with Lemma \ref{Lemma:boundsalpha}(ii), valid for $m\geq 1$. This yields
		\begin{align*}
			|\nabla U_m|^2 & = \left|\frac{d}{dr}(r^m\alpha_m(r))\right|^2|p|^2+r^{2m}|\alpha_m(r)|^2|\nabla_{\Sd}p|^2\\
			& \leq \left(m^2r^{2m-2}C_0^2+r^{2m}\frac{m^2C_1^2}{r^2}\right)|p|^2+r^{2m}C_0^2|\nabla_{\Sd}p|^2\\
			&= m^2r^{2m-2}(C_0^2+C_1^2)|p|^2+r^{2m}C_0^2|\nabla_{\Sd}p|^2.
		\end{align*}
		Hence 
		\begin{align*}
			\int_{\Bd}|\nabla U_m|^2\leq \int_0^1r^{d-1}m^2r^{2m-2}(C_0^2+C_1^2)dr\int_{\Sd}|p|^2d\sigma+\int_0^1r^{d-1}r^{2m}C_0^2dr\int_{\Sd}|\nabla_{\Sd}p|^2d\sigma\\
			\leq \frac{m^2(C_0^2+C_1^2)}{2m+d-2}\|p\|_{L^2(\Sd)}^2+\frac{C_0^2}{2m+d}\|\nabla_{\Sd}p\|_{L^2(\Sd)}^2.
		\end{align*}
		The left-hand side is finite since $2m+d-2\geq 1$. For the case $m=0$, the function $p$ is constant, and applying  Lemma \ref{Lemma:boundsalpha}(iii) we obtain
		\[
		\int_{\Bd}|\nabla U_0|^2\leq  \omega_{d-1}|p|^2(C_0^2+C_2^2)\int_0^1r^{d-1}r^{-2}dr=\omega_{d-1}|p|^2(C_0^2+C_2^2)\frac{1}{d-2},
		\]
		because $d-2\geq 0$. Hence $U_m\in W^{1,2}(\Bd)$. For any $0<\varepsilon<1$, $\alpha_m\in W^{2,1}(\varepsilon,1)$, hence $U_m\in W^{2,1}(A_{\varepsilon}(0))$. By Proposition \ref{Prop:Solbessel}, we conclude that $U_m$ satisfies Eq. \eqref{eq:radialSchr} a.e. in $A_{\varepsilon}(0)$. 
		
		Finally, let $\varphi\in \mathscr{D}(\Bd)$ and fix $0\menor \varepsilon\menor 1$. Using Green's identity on  $A_{\varepsilon}(0)$ we have
		\begin{align*}
			\int_{A_{\varepsilon}(0)}\left\{\nabla U_m\cdot \nabla \varphi+VU_m\varphi\right\} &= \int_{\Sd-S_{\varepsilon}^d(0)}\varphi \frac{\partial U_m}{\partial \nu}d\sigma +\int_{A_{\varepsilon}(0)}\left\{(-\Delta U_m+VU_m)\varphi\right\}\\
			&=  -\int_{S_{\varepsilon}^d(0)}\varphi (m\varepsilon^{m-1}\alpha_m(\varepsilon)+\varepsilon^m\alpha_m'(\varepsilon))d\sigma.
		\end{align*}
		For  $m\geq 1$, applying Lemma \ref{Lemma:boundsalpha}(i) and (ii) we obtain 
		\begin{align*}
			\left|\int_{S_{\varepsilon}^d(0)}\varphi (m\varepsilon^{m-1}\alpha_m(\varepsilon)+\varepsilon^m\alpha_m'(\varepsilon))d\sigma\right| & \leq m(C_1+ C_0)\|\varphi\|_{L^{\infty}(\Bd)}\omega_{d-1}\varepsilon^{m+d-2}\\
		\end{align*}
		Passing to the limit when $\varepsilon\rightarrow 0^+$, we obtain that $\int_{\Bd}\left\{\nabla U_m\cdot \nabla \varphi+VU_m\varphi\right\}=0$. Therefore, $U_m\in \operatorname{Sol}_V^w(\Bd)$. Since $d-2>0$, the same conclusion is valid for $m=0$.
	\end{proof}

	\begin{definition}
		A {\bf formal spherical polynomial} of degree $m$, is a function of the form $U_m(x)=r^m\alpha_m(r)p(\xi)$, where $p\in \mathcal{H}_m(\Sd)$. The collection of the formal spherical polynomials is denoted by $\mathcal{S}_m(\Bd)$.
	\end{definition}
	
	Given the orthonormal basis $\{Y_j^{(m)}\}_{j=1}^{d_m}$ of $\mathcal{H}_m(\Sd)$, we define the corresponding formal spherical polynomials  by 
	\begin{equation}
		\mathcal{V}_j^{(m)}(x):= r^m\alpha_m(r)Y_j^{(m)}(\xi),\quad j=1,\dots, d_m.
	\end{equation}
	It is worth mentioning that there exist explicit formulas for some orthonormal basis $\{Y_j^{(m)}\}_{j=1}^{d_m}$ in higher dimensions, see, e.g., \cite[Ch. 1]{day} and \cite[Remark 3.2]{mineradial1}.

	\begin{proposition}
		\begin{itemize}
			\item[(i)] If $m\neq n$, $\mathcal{S}_m(\Bd)\perp_{W^{1,2}(\Bd)} \mathcal{S}_n(\Bd)$.
			\item[(ii)] The set $\{\mathcal{V}_j^{(m)}\}_{j=1}^{m}$ forms an orthogonal basis for $\mathcal{S}_m(\Bd)$ with the $W^{1,2}$-norm. 
		\end{itemize}
	\end{proposition}
	\begin{proof}
		\begin{itemize}
			\item[(i)] Let $p\in \mathcal{H}_m(\Sd)$ and $q\in \mathcal{H}_n(\Sd)$, and let $U_m$ and $U_n$ denote their corresponding formal spherical polynomials. Hence
			\begin{align*}
				\langle U_m, U_n\rangle_{W^{1,2}(\Bd)} = & \int_{\Bd}U_m\overline{U_n}+\int_{\Bd} \frac{\partial U_m}{\partial r}\overline{\frac{\partial U_n}{\partial r}}+\int_{\Bd}\nabla_{\Sd}U_m\cdot \overline{\nabla_{\Sd}U_n}d\sigma \\
				=& \int_0^1r^{d-1}\left(r\pot{m+n}\alpha_m(r)\overline{\alpha_m(r)}+\frac{d(r^m\alpha_m(r))}{dr}\overline{\frac{d(r^n\alpha_n(r))}{dr}}\right)dr\int_{\Sd}p\overline{q}d\sigma\\ &+\int_0^1r^{d-1}r^{m+n}\alpha_m(r)\overline{\alpha_n(r)}dr\int_{\Sd}\nabla_{\Sd}p\cdot \overline{\nabla_{\Sd}q}d\sigma =0,
			\end{align*}
			where the integrals over $\Sd$ are cero by Lemma \ref{Prop:sphericalprop1}(i) and (v).
			\item[(ii)] Applying the same procedure of the point (i) and using Lemma \ref{Prop:sphericalprop1}(vi), the set $\{\mathcal{V}_j^{(m)}\}_{j=1}^{m}$ is orthogonal in $W^{1,2}(\Bd)$. Since the mapping $\mathcal{H}_m(\Sd)\ni p\mapsto r^m\phi_m(r)p\in \mathcal{S}_m(\Bd)$ is a linear isomorphism, we conclude that $\{\mathcal{V}_j^{(m)}\}_{j=1}^{m}$ is orthogonal in $W^{1,2}(\Bd)$ is an orthogonal basis for $\mathcal{S}_m(\Bd)$.
		\end{itemize}
	\end{proof}
	
	\section{Solution of the Dirichlet problem for Sobolev traces}
	
	In this section, we consider the Dirichlet problem for the radial equation \eqref{eq:radialSchr} in the sense of the Sobolev traces (DS): given a function $\varphi\in W^{\frac{1}{2},2}(\Sd)$, find a function $u\in \operatorname{Sol}_V^w(\Bd)$ such that $\operatorname{tr}_{\Sd}u=\varphi$. 
	
	In the case where $V\in L^{\infty}(\Bd)$ is real valued and satisfies the condition $\inf_{0\leq r\leq 1}V(r)> -\lambda_0$, where $\lambda_0=\inf\limits_{u\in W_0^{1,2}(\Bd)\setminus\{0\}}\frac{\||\nabla u|\|_{L^2(\Bd)}}{\|u\|_{L^2(\Bd)}}$ (the best constant for the Sobolev embedding $W_0^{1,2}(\Bd)\hookrightarrow L^2(\Bd)$, or equivalently, the first eigenvalue of the DS problem for the Laplace equation in $\Bd$), the DS problem admits a unique solution (indeed, it is not difficult to see that this conditions guarantees that the bilinear form $B[u,v]$ induces an inner product in $W_0^{1,2}(\Bd)$ that is equivalent to the usual one, and the result follows from the Lax-Milgram Theorem, see \cite[p. 294]{brezis}). In general, other conditions are required to ensure uniqueness of the problem (see, e.g., \cite[Ch. IV]{mclean}). In what follows, we will present a condition under which the problem admits at least one solution
	
	\begin{assumption}\label{assumption1}
		The potential $V\in L_1^1(0,1)$ satisfies the condition: 
		\[
		\alpha_m(1)\neq 0\quad \forall m\in \mathbb{N}_0.
		\]
		This is equivalent to stating that, for each $m\in \mathbb{N}$, $\lambda= 0$ is not an eigenvalue for the regular Sturm-Liouville problem 
		\begin{align*}
			-u''_m+\frac{\ell_m(\ell_m+1)}{r^2}u_m+V(r)u_m&=\lambda u_m,\quad 0<r<1\\
			u_m(0)=u_m(1)=0.&
		\end{align*}
		
	\end{assumption}

	\begin{remark}\label{Remark:asumption1}
		\begin{itemize}
			\item[(i)] Suppose that $V$ is non negative. Since $\mathcal{L}_{\ell}(r,s)\geq 0$ for $0<s<r\leq 1$, it follows that $\psi_k^{\ell}(r)\geq 0$. Hence, $y_m^{\ell}(r)=r^{\ell+1}+\sum_{k=1}^{\infty}\psi_k^{\ell}(r)\geq r^{\ell+1}$. Consequently, $y_m(1)> 1$, and  $\alpha_m(1)>0$. This conclusion holds, for instance, in the case of the {\bf Coulomb potential} $V(r):=\frac{c}{r}$ with $c>0$. 
			\item[(ii)] If $V\in L^{\infty}(\Bd)$ with $\inf_{0\leq r\leq 1}V(r)>-\lambda_0$, then $V$ satisfies Assumption \ref{assumption1}. Indeed, if $\alpha_m(1)=0$ for some $m\in \mathbb{N}_0$, then $U_m=r^m\alpha_m(r)p\in \mathcal{S}_m(\Bd)$ is a solution of the DS problem with $\varphi=0$, contradicting the uniqueness of the solution for this kind of potentials.
		\end{itemize}
		
	\end{remark}

	We recall that for $\varphi\in L^2(\Bd)$, its Fourier coefficients with respect to the orthonormal basis $\{\{Y_j^{(m)}\}_{j=1}^{d_m}\}_{m=0}^{\infty}$ are denoted by $\{\{\widehat{\varphi}_{m,j}\}_{j=1}^{d_m}\}_{m=0}^{\infty}$.
	
	\begin{lemma}\label{Lemma:boundedsequence}
		If $V$ satisfies the Assumption \ref{assumption1}, then the sequence $\left\{\frac{1}{\alpha_m(1)}\right\}_{m=0}^{\infty}$ is bounded.
	\end{lemma}
	\begin{proof}
		By the estimates \eqref{eq:bounpsikprime},
		\begin{align*}
			|\alpha_m(1)-1|=|y_m(1)-1|\leq \sum_{k=1}^{\infty}|\psi_k^{\ell_m}(1)|\leq \sum_{k=1}^{\infty}\frac{\left(\frac{2}{2\ell_m+1}\right)^k\|V\|_{L_1^1(0,1)}^k}{k!}=e^{\frac{2}{2\ell_m+1}\|V\|_{L_1^1(0,1)}}-1.
		\end{align*}
		Taking the limit when $m\rightarrow \infty$, we conclude that $\lim\limits_{m\rightarrow \infty}\alpha_m(1)=1$. Since $\alpha_m(1)\neq 0$ for all $m\in \mathbb{N}_0$, it follows that $\lim\limits_{m\rightarrow \infty}\frac{1}{\alpha_m(1)}=1$, and the sequence is bounded.
	\end{proof}
	\begin{theorem}
		Suppose that $V\in L_{d-1}^2(0,1)$ satisfies  Assumption \ref{assumption1}. Given $\varphi\in W^{\frac{1}{2},2}(\Bd)$, a solution $u$ of the DS problem satisfying $\operatorname{tr}_{\Sd}u=\varphi$ is given by
		\begin{equation}\label{eq:DSsolution}
			u_{\varphi}=\sum_{m=0}^{\infty}\sum_{j=1}^{d_m}\frac{\widehat{\varphi}_{m,j}}{\alpha_m(1)}\mathcal{V}_j^{(m)}
		\end{equation}
		and the series converge in the norm of $W^{1,2}(\Bd)$.

		When $V\in L^{\infty}(\Bd)$ with $\inf_{0\leq r\leq 1}V(r)\geq -\lambda_0$, this solution is unique.
	\end{theorem}
	
	\begin{proof}
		By Lemma \ref{Lemma:boundedsequence}, $C_4:= \sup\limits_{m\in \mathbb{N}_0}\frac{1}{|\phi_m(1)|}<\infty$. We now show that the series \eqref{eq:DSsolution} converges in $W^{1,2}(\Bd)$.
		
		First, note that $\|\mathcal{V}_j^{(m)}\|_{L^2(\Bd)}\leq \frac{C_0^2}{m+d}$, which implies that $$\sum_{m=0}^{\infty}\sum_{j=1}^{d_m}\left\|\frac{\widehat{\varphi}_{m,j}}{\alpha_m(1)}\mathcal{V}_j^{(m)}\right\|_{L^2(\Bd)}^2\leq C_0^2C_4^2\sum_{m=0}^{\infty}\sum_{j=1}^{d_m}|\widehat{\varphi}_{m,j}|^2<\infty,$$ and the series converges in $L^2(\Bd)$.
		
		For the gradient, from the proof of Theorem \ref{Th:Sol1} we obtain the following estimate 
		\begin{align*}
			\|\nabla\mathcal{V}_j^{(m)}\|_{L^2(\Bd)}^2&\leq \frac{m^2(C_0^2+C_1^2)}{2m+d-2}\|Y_j^{(m)}\|_{L^2(\Sd)}^2+\frac{C_0^2}{2m+d}\|\nabla_{\Sd}Y_j^{(m)}\|_{L^2(\Sd)}^2\\
			&= \frac{m^2(C_0^2+C_1^2)}{2m+d-2}+\frac{C_0^2}{2m+d}m(m+d-2).
		\end{align*}
		In the last equality, we use \eqref{eq:normgradspherical}. Since $\lim\limits_{m\rightarrow\infty}\frac{m}{2m+d-2}=1$, there exists a constant $C_5$ with $\frac{m}{2m+d}\leq \frac{m}{m+d-2}\leq C_5$. Hence  
		\begin{align*}
			\|\nabla\mathcal{V}_j^{(m)}\|_{L^2(\Bd)}^2 &\leq (C_0^2+C_1^2)C_5 m+ (m+d-2)C_5C_0^2.
		\end{align*} 
		Again, since $\lim\limits_{m\rightarrow \infty}\frac{m}{m+d-2}=1$, we can find constants $C_6, C_7>0$ such that $C_7\leq \frac{m}{m+d-2}\leq C_6$. Combining these estimates, we get
		\[
		\sum_{m=0}^{\infty}\sum_{j=1}^{d_m}\left\|\frac{\widehat{\varphi}_{m,j}}{\alpha_m(1)}\nabla \mathcal{V}_j^{(m)}\right\|_{L^2(\Bd)}^2 \leq \bar{C}\sum_{m=0}^{\infty}\sum_{j=1}^{d_m}\sqrt{m(m+d-2)}|\widehat{\varphi_{m,j}}|^2<\infty,
		\]
		for some constant $\bar{C}\mayor 0$ depending only on $C_0,\dots, C_7$. Thus, the series \eqref{eq:DSsolution} converges in $W^{1,2}(\Bd)$ and $u_{\varphi}\in W^{1,2}(\Bd)$. By Proposition \ref{Prop:closedsubspace}, we conclude that $u_{\varphi}\in \operatorname{Sol}_V^d(\Bd)$.
		
		Finally, define the partial sums $u_N=\sum_{m=0}^{N}\sum_{j=1}^{d_m}\frac{\widehat{\varphi}_{m,j}}{\alpha_m(1)}\mathcal{V}_j^{(m)}$. Since each $\mathcal{V}_j\pot{(m)}\in C(\overline{\mathbb{B}^d})$, it follows that  $\operatorname{tr}_{\Sd}u_N=\sum_{m=0}^{N}\sum_{j=1}^{d_m}\widehat{\varphi}_{m,j}Y_j^{(m)}=\varphi_N$. Consequently,  $\varphi_N\rightarrow \varphi$ in $L^2(\Sd)$. But $u_N\rightarrow u_{\varphi}$ in $W^{1,2}(\Bd)$ and the continuity of the trace operator implies that $\operatorname{tr}_{\Sd}u_N\rightarrow \operatorname{tr}_{\Sd}u$. Therefore, $\operatorname{tr}_{\Sd}u_{\varphi}=\varphi$.
	\end{proof}
	
	From the estimates obtained in the proof of Theorem \ref{Th:Sol1}, we deduce:
	
	\begin{corollary}
		The operator $W\pot{\frac{1}{2},2}(\Sd)\ni \varphi\mapsto u_{\varphi}\in \operatorname{Sol}_V\pot w(\Bd)$ is bounded.
	\end{corollary}
	
	\begin{corollary}
		If $V\in L^{\infty}(0,1)$ and satisfies $\inf_{0\leq r\leq 1}V(r)> -\lambda_0$, then $\{\{\mathcal{V}_j^{(m)}\}_{j=1}^{d_m}\}_{m=0}^{\infty}$ is an orthogonal basis for $\operatorname{Sol}_V^w(\Bd)$.
	\end{corollary}
	\begin{proof}
		Given $u\in \operatorname{Sol}_V^w$, this is the unique solution to the DS problem with $\varphi=\operatorname{tr}_{\Sd}u$. Hence $u$ admits the Fourier-series \eqref{eq:DSsolution}.  
	\end{proof}
	
	Now, we establish a condition ensuring uniqueness for the DS problem.
	
	\begin{definition}
		A potential $V$ is said to be {\bf almost bounded from below}, if there exists a constant $C>0$ such that $\operatorname{Re}V(r)\geq -C$ a.e. in $(0,1)$.
	\end{definition}
	
	For example, the  Coulomb potential $V(r)=\frac{c}{r}$ with $\operatorname{Re}c>0$ is almost bounded from below, but does not belong to $L^{\infty}(0,1)$.

	\begin{theorem}
		Suppose that $V\in L_{d-1}^r(0,1)$, where $r=\max\{2,\frac{d}{2}\}$, is almost bounded from below. Then, the DS problem admits a unique solution iff $\operatorname{tr}_{\Sd}:\operatorname{Sol}_V^w(\Bd)\rightarrow W^{\frac{1}{2},2}(\Bd)$ is surjective.
	\end{theorem}
	
	\begin{proof}
		By Proposition \ref{Prop:closedsubspace}, the bilinear form $B:W_0\pot{1,2}(\Bd)\times W_0\pot{1,2}(\Bd)\rightarrow \mathbb{C}$ given by \eqref{eq:weaksol2ndform}, is bounded. This allows us to define an operator $\mathbf{A}: W_0^{1,2}(\Bd)\rightarrow W^{-1,2}(\Bd)$, given by $\mathbf{A}u=B[u,\cdot]$. Since the embedding $W_0^{1,2}(\Bd)\hookrightarrow L^2(\Bd)$ is continuous and  $W_0^{1,2}(\Bd)$ is a dense subspace of $L^2(\Bd)$, we can identify $W^{1,2}_0(\Bd)\subset L^2(\Bd)\subset W^{-1,2}(\Bd)$. So, we have that $L^2(\Bd)$ is a pivot space for $W_0^{1,2}(\Bd)$ (see \cite[p. 44]{mclean}). Hence,  using that $V$ is almost bounded from below and the Poincaré inequality \cite[Cor. 9. 19]{brezis}, we obtain
		\begin{align*}
			\operatorname{Re}B[u,u] =\int_{\Bd}|\nabla u|^2+\int_{\Bd}\operatorname{Re}(V)|u|^2\geq \|u\|_{W^{1,2}(\Bd)}^2-C\|u\|_{L^2(\Bd)}^2.
		\end{align*}
		This shows that $B$ is coercive in $W_0^{1,2}(\Bd)$ relative to the pivot space $L^2(\Bd)$. Moreover, since the embedding $W^{1,2}_0(\Bd)\hookrightarrow L^2(\Bd)$ is compact \cite[Th. 9. 16]{brezis}, it follows from \cite[Th. 2. 34]{mclean} that the operator $\mathbf{A}: W_0^{1,2}(\Bd)\hookrightarrow W^{-1,2}(\Bd)$ is Fredholm with index $0$. In particular, the space $W$ of solutions in $\operatorname{Sol}_V^w(\Bd)$ with zero trace is finite dimensional. Suppose that $\{w_1,\dots, w_N\}$ is a basis for $W$.  According to \cite[p. 115]{mclean}, the adjoint operator of $\mathbf{A}$ is given by 
		\[
		(\mathbf{A}^{\star}u)v= \int_{\Bd} \nabla u\cdot \nabla v+\overline{V}uv,\quad u,v\in W_0^{1,2}(\Bd).
		\]
		Note that $w_j$ is a solution of trace zero of the Schr\"odinger equation with potential $V$ iff $\overline{w}_j$ is a solution of the corresponding equation with potential $\overline{V}$. By the Fredholm alternative \cite[Th. 2. 27]{mclean}, the non-homogeneous equation $\mathbf{A}u=f$, $f\in W^{-1,2}(\Bd)$, admits a solution $u\in W_0^{1,2}(\Bd)$ iff $(f|\overline{w_j})_{W^{1,2}_0(\Bd)}=0$, $j=1,\dots, N$. Now, consider the DS problem for the Schr\"odinger equation with potential $V$ with Dirichlet data $\varphi\in W^{\frac{1}{2},2}(\Bd)$. It is easy to see that the DS problem is equivalent to solve $\mathbf{A}y=f$ with $f=\mathbf{A}\mathcal{E}\varphi$, where $\mathcal{E}\varphi\in W^{1,2}(\Bd)$ is an extension of $\varphi$. Hence, the DS problem has a unique solution iff
		\begin{align}
			0=(\mathbf{A}\mathcal{E}\varphi|\overline{w}_j)_{W_0^{1,2}(\Bd)}= \int_{B}\nabla \mathcal{E}\varphi \cdot \nabla \overline{w}_j+V\mathcal{E}\varphi \overline{w}_j \label{eq:generalizednormalderivative}.
		\end{align} 
		The right-hand side of \eqref{eq:generalizednormalderivative} defines a functional in $W^{-\frac{1}{2},2}(\Sd):= \left(W^{\frac{1}{2},2}(\Sd)\right)'$, which is independent of the choice of extension $\mathcal{E}\varphi$ (see \cite[Lemma 4. 3]{mclean} and \cite[Ch. 7]{mitesis}). We call this functional the {\it generalized normal derivative of } $\overline{w}_j$, denoted by $\frac{\partial \overline{w}_j}{\partial \nu}$, $j=1,\dots, N$. In particular, if $\varphi=\operatorname{tr}_{\Sd} u$ with $u\in \operatorname{Sol}_V^w(\Bd)$, then
		\[
		\left( \frac{\partial \overline{w}_j}{\partial \nu}\bigg{|} \operatorname{tr}|_{\Sd}u\right)_{W^{\frac{1}{2},2}(\Sd)} = \int_{\Bd} \nabla u\cdot \nabla \overline{w}_j+Vu\overline{w}_j=0,\quad j=1,\dots, N,
		\]
		because $u\in \operatorname{Sol}_V^w(\Bd)$ and $\overline{w}_j\in W_0^{1,2}(\Bd)$. Conversely, if $\varphi \in W^{\frac{1}{2},2}(\Sd)$ satisfies that $\left( \frac{\partial \overline{w}_j}{\partial \nu}\bigg{|} \varphi\right)_{W^{\frac{1}{2},2}(\Sd)} =0$, then Eq.  \eqref{eq:generalizednormalderivative} implies that there exists $y\in W_0^{1,2}(\Bd)$ such that $\mathbf{A}y=\mathbf{A}\mathcal{E}\varphi$, i.e., there is $u\in \operatorname{Sol}_V^w(\Bd)$ with $\operatorname{tr}_{\Sd}u=\varphi$.  Thus, the pre-annihilator of $E=\operatorname{Span}\left\{\frac{\partial \overline{w}_j}{\partial \nu}\right\}_{j=1}^N$ in $W^{\frac{1}{2},2}(\Sd)$ is precisely $\operatorname{tr}_{\Sd}\left(\operatorname{Sol}_V^w(\Bd)\right)$. Since $E$ is finite dimensional, by \cite[Ch. 2, Lemma 1]{mitesis}, the annihilator of $\operatorname{Sol}_V^w(\Bd)$ is $E$. 
		
		Finally, $\operatorname{tr}_{\Sd}\left(\operatorname{Sol}_V^w(\Bd)\right)=W^{\frac{1}{2}}(\Sd)$ iff $E=\{0\}$ which is equivalent to the DS problem having a unique solution.
	\end{proof}
	
	\begin{corollary}\label{Cor:basis}
		If $V\in L_{d-1}^r(0,1)$ with  $r=\max\{2,\frac{d}{2}\}$, is almost bounded from below, then the DS problem admits a unique solution iff the potential $V$ satisfies Assumption \ref{assumption1}. In this case, the family $\{\{\mathcal{V}_j\pot{(m)}\}_{j=1}\pot{d_m}\}_{m=0}\pot{\infty}$ forms an orthogonal basis for $\operatorname{Sol}_V\pot{w}(\Bd)$.
	\end{corollary}

	\section{The generalized Poisson kernel}
	
	In this section, we assume that the potential $V$ satisfies Assumption \ref{assumption1}. In the case when $\varphi\in L^2(\Sd)$, the series \eqref{eq:DSsolution} converges to $u_{\varphi}$ in $L^2(\Bd)$ and $u_{\varphi}\in \operatorname{Sol}_V^{dist}(\Bd)$. However, in general, the series does not converge in $W^{1,2}(\Bd)$. Nevertheless, for $x=r\xi$, we may informally manipulate the series to obtain an integral representation. Specifically, using the Fourier expansion of $\varphi$ we write:
	
	\begin{align*}
		u_{\varphi}(r\xi) & =  \sum_{m=0}^{\infty}\sum_{j=1}^{d_m}\frac{r^m\alpha_m(r)}{\alpha_m(1)}\left(\int_{\Sd}\varphi(\zeta)\overline{Y_j^{(m)}(\zeta)}d\sigma_{\zeta}\right)Y_j^{(m)}\\
		&= \int_{\Sd}\varphi(\zeta)\left(\sum_{m=0}^{\infty}\sum_{j=1}^{d_m}\frac{r^m\alpha_m(r)}{\alpha_m(1)}Y_j^{(m)}(\xi)\overline{Y_j^{(m)}(\zeta)}\right)d\sigma_{\zeta},
	\end{align*}
	The series can be rewritten as follows:  
	\[
	\sum_{m=0}^{\infty}\sum_{j=1}^{d_m}\frac{r^m\alpha_m(r)}{\alpha_m(1)}Y_j^{(m)}(\xi)\overline{Y_j^{(m)}(\zeta)}=\sum_{m=0}^{\infty}\frac{r^m\alpha_m(r)}{\alpha_m(1)}Z_m(\xi,\zeta),
	\]
	where 
	\begin{equation}\label{eq:zonalharmonic}
		Z_m(\xi,\zeta):=\sum_{j=1}^{d_m}Y_j^{(m)}(\xi)\overline{Y_j^{(m)}(\zeta)}.
	\end{equation}
	The function $Z_m(\xi,\zeta)$ is called the {\it zonal harmonic} of degree $m$, and is precisely the reproducing kernel for the finite-dimensional space $\mathcal{H}_m(\Sd)$ with the $L^2$-inner product (that is, $p(\xi)=\int_{\Sd}p(\xi)Z_m(\xi,\zeta)d\sigma_{\zeta}$ for all $p\in \mathcal{H}_m(\Sd)$, $\xi\in \Sd$). The zonal harmonic is uniquely determined by its reproducing property and is independent of the choice of the orthonormal basis $\{Y_j^{(m)}\}_{j=1}^{d_m}$. The following proposition summarizes the main properties of the Zonal harmonics. A detailed proof can be found in \cite[Prop. 5.37]{axler}. 
	
	\begin{proposition}\label{Prop:zonal}
		\begin{itemize}
			\item[(i)] The zonal harmonic is real valued and symmetric: $Z_m(\xi,\zeta)=Z_m(\zeta,\xi)$.
			\item[(ii)] If $m\neq n$, then $\int_{\Sd}Z_m(\xi,\zeta)Z_n(\xi,\zeta)d\sigma_{\xi}=0$ for all $\xi\in \Sd$.
			\item[(iii)] $\int_{\Sd}Z_m^2(\xi,\zeta)d\sigma_{\zeta}=Z_m(\xi,\xi)=d_m$ for all $\xi\in \Sd$.
			\item[(iv)] $|Z_m(\xi,\zeta)|\leq Z_m(\xi,\xi)$ for all $\xi,\zeta\in \Sd$.
			\item[(v)] There exists a constant $C\pot d$ depending only on the dimension $d$ and satisfying $Z_m(\xi,\xi)\leq C\pot d m^{d-1}$ for all $\xi \in \Sd$. 
		\end{itemize}
	\end{proposition}
	\begin{remark}\label{Remark:zonalexplicit}
		According to \cite[Th. 5.38]{axler}, the zonal harmonic $Z_m(\xi,\zeta)$ admits the expression 
		\begin{equation*}
			Z_m(\xi,\zeta)=(d+2m-2)\sum_{k=0}^{\left[\frac{m}{2}\right]}(-1)^k\frac{d(d+2)\cdots(d+2m-2k-4)}{2^kk!(m-2k)!}(\xi\cdot \zeta)^{m-2k}.
		\end{equation*}
		We recall that the Pochhammer symbol is defined by $(z)_n:=z(z+1)\cdots(z+n-1)=\frac{\Gamma(z+n)}{\Gamma(z)}$. Hence,
		\begin{align*}
			d(d+2)\cdots(d+2m-2k-4)&=2^{m-k-1}\left(\frac{d}{2}\right)\left(\frac{d}{2}+1\right)\cdots \left(\frac{d}{2}-1+m-k-1\right)\\
			&=2^{m-k-1}\left(\frac{d}{2}-1\right)_{m-k}=2^{m-k-1}\frac{\Gamma\left(\frac{d}{2}-1+m-k\right)}{\Gamma\left(\frac{d}{2}-1\right)}.
		\end{align*}
		Thus,
		\begin{align*}
			Z_m(\xi,\zeta)&=(d+2m-1)\sum_{k=0}^{\left[\frac{m}{2}\right]}(-1)^k\frac{2^{m-k-1}\Gamma\left(\frac{d}{2}-1+m-k\right)}{2^kk!\Gamma\left(\frac{d}{2}-1\right)}(\xi\cdot \zeta)^{m-2k}\\
			&= \left(\frac{d}{2}+m-1\right)\sum_{k=0}^{\left[\frac{m}{2}\right]}(-1)^k\frac{\Gamma\left(\frac{d}{2}-1+m-k\right)}{k!\Gamma\left(\frac{d}{2}-1\right)}(2\xi\cdot \zeta)^{m-2k}=\left(m+\frac{d}{2}-1\right)C_m^{\left(\frac{d}{2}-1\right)}(\xi\cdot \zeta),
		\end{align*}
		where $C^{(\alpha)}_m(t)$ denotes the Gegenbauer polynomial of degree $m$ and order $\alpha>-\frac{1}{2}$ (see \cite[Sec. 22. 1]{abramowitz}).
	\end{remark}
	\begin{definition}
		The {\bf generalized Poisson kernel} for the radial Schr\"odinger equation is defined by
		\begin{equation}\label{eq:generalpoisson}
			P_V(r,\xi,\zeta)=\sum_{m=0}^{\infty}\frac{r^m\alpha_m(r)}{\alpha_m(1)}Z_m(\xi,\zeta),\quad 0\leq r<1, \xi,\zeta\in \Sd.
		\end{equation}
	\end{definition}
	For $V\equiv 0$, we obtain the classical Poisson kernel for harmonic functions \cite[p. 122]{axler}:
	\[
	P_0(r,\xi,\zeta)=\frac{1-r^2}{|r\xi-\zeta|^{d}}.
	\]
	Note that $P_V$ is symmetric in the spherical variables.
	
	\begin{proposition}\label{Prop:poissonconvergence}
		The series \eqref{eq:generalpoisson} converges absolutely and uniformly for $|\rho|\leq r<1, \xi,\zeta \in \Sd$. 
	\end{proposition}
	\begin{proof}
		Let $r<1$ and fix $|\rho|\leq r$, $\xi,\zeta\in \Sd$. By proposition \ref{Prop:zonal} (iv)-(v) and Lemma \ref{Lemma:boundsalpha}
		we get the estimate
		\begin{align*}			\sum_{m=0}^{\infty}\left|\frac{\rho^m\alpha_m(\rho)}{\alpha_m(1)}\right||Z_m(\xi,\zeta)|\leq C_0C_4C\pot d\sum_{m=0}^{\infty}m^{d-1}r^m.
		\end{align*}
		The series converges by the ratio test. Hence, by the Weierstrass $M$-tests,  the series converges absolutely and uniformly  for $|\rho|\leq r, \xi,\zeta\in \Sd$.
	\end{proof}
	
	As a consequence, for $r<1$, the uniform convergence ensures that
	\begin{equation}\label{eq:poissonreprensetation}
		u_{\varphi}(r\xi)=\int_{\Sd}\varphi(\xi)P_V(r,\xi,\zeta)d\sigma_{\zeta}.
	\end{equation}
	Given $\varphi\in L^2(\Sd)$,  for every $0<r<1$, the corresponding solution $u_{\varphi}$ given by \eqref{eq:poissonreprensetation}, defines a function $\varphi_r\in L^2(\Sd)$ by the relation $\varphi_r(\xi):=u_{\varphi}(r\xi)$.  The interesting question is in what sense $\varphi_r\rightarrow \varphi$ when $r\rightarrow 1^{-1}$. A first answer that arises is in the case when $\varphi\in \mathscr{D}(\Sd)$.
	
	\begin{proposition}\label{Prop:convergencetestfunct}
		If $\varphi\in \mathscr{D}(\Bd)$, the series \eqref{eq:DSsolution} converges uniformly on $\overline{\Bd}$ and $u\in C(\overline{\Bd})$. In consequence, $\varphi_r\rightarrow \varphi$ uniformly on $\Sd$.
	\end{proposition}
	\begin{proof}
		According to \cite[Sec. 4, Th. 2]{sphericaluniform}, if $\varphi\in C^p(\Sd)$ with $p>\frac{d+3}{2}$, then its Fourier series satisfies $\sum_{m=0}^{\infty}\sum_{j=1}^{d_m}|\widehat{\varphi}_{m,j}Y_j^{(m)}(\xi)|\leq C_{p,d,\varphi}$ for all $\xi\in \Sd$, where the constant $C_{p,d,\varphi}>0$ depends only on $d,p$ and $\varphi$, and the series converges uniformly and absolutely to $\varphi$. Hence, for $r\leq 1$ and $\xi\in \Sd$,
		\[
		\sum_{m=0}^{\infty}\left|\frac{r^m\alpha_m(r)}{\alpha_m(1)}\widehat{\varphi}_{m,j}Y_j^{(m)}(\xi)\right|\leq C_0C_4\sum_{m=0}^{\infty}\sum_{j=1}^{d_m}|\widehat{\varphi}_{m,j}Y_j^{(m)}(\xi)|<C_0C_4C_{p,d,\varphi}.
		\]
		So,  the series of $u_{\varphi}$ converges absolutely and uniformly on $\mathbb{B}^d$. Consequently,  $u_{\varphi}\in C(\overline{\Bd})$, hence $u_{\varphi}$ is uniformly continuous on $\overline{\Bd}$. Therefore, $\varphi_r\rightarrow \varphi$ uniformly on $\Sd$.
	\end{proof}

	For a general $\varphi\in L^2(\Sd)$, following \cite{estrada}, we introduce the concept of {\it generalized boundary value}.
	
	\begin{definition}
		A locally integrable function $u$ defined on $\Bd$ is said to have a {\bf distributional boundary value}, if there exists a distribution $f\in \mathscr{D}'(\Sd)$ such that $u(r\cdot)\rightarrow f$ in $\mathscr{D}'(\Sd)$ as $r\rightarrow 1\pot{-}$, that is,
		\[
		\lim_{r\rightarrow 1^{-}}\int_{\Sd}u(r\xi)\varphi(\xi)d\sigma_{\sigma} = (f|\varphi)_{\mathscr{D}(\Sd)}, \quad \forall \varphi\in \mathscr{D}(\Bd).
		\]
	\end{definition}
	Since the weak limit in $\mathscr{D}'(\Sd)$ is unique, we denote the distributional boundary value of $u$ by $u_{db}$.
	
	We recall that $(C(\Sd))'$ denotes the space of distributions on $C(\Sd)$. Since $\mathscr{D}(\Sd)\subset C(\Sd)$, it follows that $(C(\Sd))'\subset \mathscr{D}'(\Sd)$. By the Riesz representation theorem, $(C(\Sd))'$ can be identified with the space of complex Radon measures on $\Sd$ \cite[Cor. 7. 18]{folland}. We denote $\mathcal{M}(\Sd):=(C(\Sd))'$. It is important to note that for $f\in \mathcal{M}(\Sd)$ and $\varphi\in \mathscr{D}(\Sd)$, we have $(f|\varphi)_{\mathscr{D}(\Sd)}=(f|\varphi)_{C(\Sd)}$.
	
	\begin{theorem}\label{Th:Sol2dist}
		If $f\in \mathcal{M}(\Sd)$, then the function
		\begin{equation}\label{eq:soldistributional}
			u_f(r\xi):= (f|P_V(r,\xi,\cdot))_{C(\Sd)},\quad x=r\xi\in \Bd,
		\end{equation}
		is a distributional solution of \eqref{eq:radialSchr} which belongs to $C(\Bd)$ and $(u_f)_{db}=f$. In particular, this is valid for $f\in L^2(\Sd)$, and in this case, $u_f\in L^2(\Bd)$.
	\end{theorem}
	
	\begin{proof}
		For fixed $x=r\xi \in \Bd$, Proposition \ref{Prop:poissonconvergence} implies that $P_V(r,\xi,\cdot)\in C(\Sd)$, so the operation \eqref{eq:soldistributional} is well defined. Furthermore, since the series \eqref{eq:generalpoisson} converges in $C(\Sd)$ with respect to the variable $\zeta$, we have 
		\begin{align*}
			u_f(x)=(f|P_V(r,\xi,\cdot))_{C(\Sd)}= \sum_{m=0}^{\infty}\frac{r^m\alpha_m(r)}{\alpha_m(1)}(f|Z_m(\xi))_{\mathscr{D}(\Sd)}.
		\end{align*}
		(we recall that $(f|Z_m(\xi))_{\mathscr{D}(\Sd)}=(f|Z_m(\xi))_{C(\Sd)}$). Denote  $\widehat{f}_m(\xi):=(f|Z_m(\xi,\cdot))_{\mathscr{D}(\Sd)}=\sum_{j=1}^{d_m}(f|Y_j^{(m)})_{\mathscr{D}(\Sd)}Y_j^{(m)}(\xi)$, where $\{Y_j^{(m)}\}_{j=1}^{d_m}$ is real-valued. According to \cite[Lemma 2]{estrada}, we have $\widehat{f}_m\in \mathcal{H}_m(\Sd)$ and there exist constants $C>0$ and $\beta\in \mathbb{R}$ such that 
		\begin{equation}\label{eq:cotadistf}
			\|\widehat{f}_m\|_{L^{\infty}(\Sd)}\leq Cm\pot{\beta}.
		\end{equation}
		Thus, 
		\[
		\sum_{m=0}^{\infty}\left|\frac{r^m\alpha_m(r)}{\alpha_m(1)}\widehat{f}_m\right|\leq C_0C_4C\sum_{m=0}^{\infty}m^{\beta}r^m.
		\]
		Again, the ratio test implies that the series converges absolutely and uniformly on compact subsets of $\Bd$, which ensures that $u_f\in C(\Bd)\subset L_{loc}\pot 1(\Bd)$. Note that each term $\frac{r^m\alpha_m(r)}{\alpha_m(1)}\widehat{f}_m$ belongs to $ \mathcal{S}_m(\Bd)$. Now, let $\phi\in \mathscr{D}(\Bd)$, and choose $0<r<1$ such that $\operatorname{supp}\phi\in B_r\pot d(0)$. Since the series defining $u_f$ converges uniformly on $B_r\pot{d}(0)$, we have
		\[
		\int_{\Bd} u_f(-\Delta\phi+V\phi)=\sum_{m=0}^{\infty}\int_{B_r(0)}\frac{\rho^m\alpha_m(\rho)}{\alpha_m(1)}\widehat{f}_m(-\Delta\phi+V\phi)=0,
		\]
		Thus, $u_f\in \operatorname{Sol}_V^{dist}(\Bd)$.
		
		Finally, for $0<r<1$, set $f_r(\xi)=u_f(r\xi)$. For any test function $\varphi\in \mathscr{D}(\Sd)$ we have
		\begin{align*}
			\int_{\Sd}f_r(\xi)\varphi(\xi)d\sigma_{\xi}= \sum_{m=0}^{\infty}\frac{r^m\alpha_m(r)}{\alpha_m(1)}\int_{\Sd}\widehat{f}_m(\xi)\varphi(\xi)d\sigma_{\xi}.
		\end{align*}
		Note that
		\begin{align*}
			\int_{\Sd}\widehat{f}_m(\xi)\varphi(\xi)d\sigma_{\xi}&=\sum_{j=1}^{d_m}\int_{\Sd}(f|Y_j^{(m)})_{\mathscr{D}(\Sd)}Y_j^{(m)}(\xi)\varphi(\xi)d\sigma_{\xi}\\
			&=\sum_{j=1}^{dm}\widehat{\varphi}_{m,j}(f|Y_j^{(m)})_{\mathscr{D}(\Sd)}=\left(f\bigg{|}\sum_{j=0}^{d_m}\widehat{\varphi}_{m,j}Y_j^{(m)}\right)_{\mathscr{D}(\Sd)}.
		\end{align*} 
		Thus, 
		\[
		\int_{\Sd}f_r(\xi)\varphi(\xi)d\sigma_{\xi}=\sum_{m=0}^{\infty}\left(f\Bigg{|}\frac{r^m\alpha_m(r)}{\alpha_m(1)}\sum_{j=0}^{d_m}\widehat{\varphi}_{m,j}Y_j^{(m)}\right)_{\mathscr{D}(\Sd)}=\sum_{m=0}^{\infty}\left(f\Bigg{|}\frac{r^m\alpha_m(r)}{\alpha_m(1)}\sum_{j=0}^{d_m}\widehat{\varphi}_{m,j}Y_j^{(m)}\right)_{C(\Sd)}.
		\]
		By Proposition \ref{Prop:convergencetestfunct}, the series of the solution $u_{\varphi}$ converges uniformly to a continuous function on $\Sd$ for every $0<r<1$, that is, in $C(\Sd)$. Hence,
		\[
		\int_{\Sd}f_r\varphi d\sigma_{\xi}=(f|\varphi_r)_{C(\Sd)},\quad \forall 0<r<1.
		\]
		Since $u_{\varphi}$ is uniformly continuous on $\overline{\Bd}$, it follows that $\varphi_r\rightarrow \varphi$ in $C(\Sd)$. Thus,
		\[
		\lim_{r\rightarrow 1^{-}}\int_{\Sd}f_r\varphi d\sigma_{\xi}=(f|\varphi)_{C(\Sd)}=(f|\varphi)_{\mathscr{D}(\Sd)}.
		\]
		Therefore, $(u_f)_{db}=f$. 
		
		In the particular case where $f\in L^2(\Sd)$, the orthogonality of $\{\{\mathcal{V}_j^{(m)}\}_{j=1}^{d_m}\}_{m=0}^{\infty}$ together with Lemma \ref{Lemma:boundsalpha} implies that $u_f\in L^2(\Bd)$.
	\end{proof}

	\begin{example}
		Consider the case $d=3$ with the Coulomb potential $V(r)=\frac{c}{r}$ with $c>0$. In this case, the perturbed Bessel equation \eqref{eq:Besselperturbed} becomes
		\[
		-y''_m+\frac{m(m+1)}{r^2}y_m+\frac{c}{r}y_m=0,\quad 0<r<1.
		\]
		A direct computation using the method from Proposition \ref{Prop:Solbessel} shows that the solutions take the form $\psi_k^{(m)}(r)=r^{m+1}\frac{(cr)^k}{k!(2m+2)_k}$, $k\in \mathbb{N}_0$. Thus, $y_m(r)=r^{m+1}\sum_{k=0}^{\infty}\frac{(cr)^k}{k!(2m+2)_k}$. Note that
		\begin{align*}
			\alpha_m(r)&=\sum_{k=0}^{\infty}\frac{(cr)^k}{k!(2m+2)_k}=\Gamma (2m+2)(\sqrt{cr})^{-2m-1}\sum_{m=0}^{\infty}\frac{(\sqrt{cr})^{2k+2m+1}}{k!\Gamma(k+(2m+1)+1)}\\
			&=\frac{(2m+1)!}{(cr)^{m+\frac{1}{2}}}I_{2m+1}(2\sqrt{cr}),            
		\end{align*}
		where $I_{\nu}(z)$ denotes the modified Bessel function of order $\nu$. Applying the classical formula for  spherical harmonics in $\mathbb{R}^3$, we obtain the following system of orthogonal solutions
		\begin{equation}\label{eq:basiscoulomb}
			\mathcal{V}_j^{(m)}(r,\theta_1,\theta_2)=\frac{(2m+1)!}{\sqrt{r}c^{\frac{m+1}{2}}}\sqrt{\frac{(2m+1)(m-j)!}{4\pi (m+j)!}}I_{2m+1}(r)P_j^{(m)}(\cos \theta_1)e^{im\theta_2}
		\end{equation}
		for  $m\in \mathbb{N}_0, \; j=-m,\dots, m$, where $P_j^{(m)}(x)=\frac{(-1)^j}{2^mm!}(1-x^2)^\frac{j}{2}\frac{d^{m+j}}{d^{m+j}}(1-x^2)^m$ (the associated Legendre polynomial). 	For $c>0$, according to Remark \ref{Remark:asumption1}(i), the Coulomb potential satisfies Assumption \ref{assumption1}. Since $V\in L_2^2(0,1)$, by Corollary \ref{Cor:basis}, the system \eqref{eq:basiscoulomb} is an orthogonal basis in $W^{1,2}(\mathbb{B}^2)$ for $\operatorname{Sol}_V^w(\mathbb{B}^2)$.
		
		By Remark \ref{Remark:zonalexplicit}, the Poisson kernel admits the following expression in the spherical coordinates $0<r<1$, $0<\theta_1,\vartheta_1<\pi$, $0<\theta_2,\vartheta_2<2\pi$:
		\begin{align*}
			P_V(r,\theta_1,\theta_2,\vartheta_1,\vartheta_2)=&\sum_{m=0}^{\infty}\frac{(2m+1)I_{m+1}(2\sqrt{cr})}{2\sqrt{r}I_{m+1}(2\sqrt{c})}\cdot  C_m^{\left(\frac{1}{2}\right)}\left(\sin\theta_2\sin\vartheta_2\cos(\theta_1-\vartheta_1)+\cos\theta_2\cos\vartheta_2\right).
		\end{align*}
		Consider the distributional boundary value problem with $f(\xi):=\delta_N(\xi):=\delta(\xi-N)$, where $N=(0,0,1)$ and $\delta\in \mathcal{M}(\mathbb{S}^2)$ is the Dirac delta distribution. It is known that $Z_m(\xi,N)=P_m(\cos(\theta_1))$, where $P_m(t)$ denotes the Legendre polynomial of degree $m$, normalized so that $P_m(1)=1$ \cite[p. 302]{witacker}. The zonal harmonic $Z_m(\xi,N)$ satisfies the property of invariance under isometries that fix $N$ \cite[p. 101]{axler}. Hence, the solution $u_{\delta_N}$ of \eqref{eq:radialSchr} with the Coulomb potential and distributional boundary value $\delta_N$ is given by
		\begin{equation*}
			u_{\delta_N}(r,\theta_1,\theta_2)= \sum_{m=0}^{\infty}\frac{I_{m+1}(2\sqrt{cr})}{\sqrt{r}I_{m+1}(2\sqrt{c})} P_m(\cos\theta_1).
		\end{equation*}
	\end{example}

	\section{The case $d=2$}\label{Sec:d=2}
	
	For the case $d=2$, the construction of the functions $\{\alpha_m\}_{m=0}\pot{\infty}$ is exactly the same: $\alpha_m(r)=\frac{y_m(r)}{r\pot{\ell_m+1}}$, where $y_m=w_{\ell_m}$ with $\ell_m=m-\frac{1}{2}\mayor 0$ for case $m\in \mathbb{N}$, and the statements (i) and (ii) of Lemma \ref{Lemma:boundsalpha} remain valid. For the case $m=0$,  $\alpha_0(r)=\frac{y_0(r)}{\sqrt{r}}$, where $y_0$ is the unique solution of the perturbed Bessel equation
	\begin{equation}\label{eq:besselsingular}
		-y_0''(r)-\frac{1}{4r\pot 2}y_0(r)+V(r)y_0(r)=0,\quad 0\menor r\menor 1,
	\end{equation}
	satisfying the asymptotic
	\begin{equation}\label{eq:asymptsingular}
		y_0(r)\sim \sqrt{r},\quad y_0'(r)\sim \frac{1}{2\sqrt{r}},\quad r\rightarrow 0\pot +.
	\end{equation}
	
	The existence of such a solution is ensured under the additional hypotheses given in \cite{kostenko}
	\begin{equation}\label{eq:hyphotesesTeschl}
		\int_0^1r(1-\log r)|V(r)|dr<\infty.
	\end{equation}
	For a more practical construction, in \cite{mineradial1} an explicit method was proposed using the same procedure as in Proposition \ref{Prop:Solbessel}. We extend this result for the hypotheses that $V\in L^1(0,1)$.
	
	\begin{proposition}
		Suppose that $V\in L^1(0,1)$. Then the unique solution $y_0(r)$ of Eq. \eqref{eq:besselsingular} satisfying the asymptotic conditions \eqref{eq:asymptsingular} is given by
		\begin{equation}\label{eq:seriesbesselsingular}
			y_0(r)=\sum_{k=0}^{\infty}\psi_k(r),
		\end{equation}
		where 
		\begin{equation}
			\psi_k(r):= \begin{cases}
				\sqrt{r},& \text{ if } k=0,\\
				\int_0^r\mathcal{L}_{-\frac{1}{2}}(r,s)V(s)\psi_{k-1}(s)ds, & \text{ if } k\geq 1,
			\end{cases}
		\end{equation}
		with $\mathcal{L}_{-\frac{1}{2}}(r,s)=\sqrt{rs}\log\left(\frac{r}{s}\right)$. The series \eqref{eq:seriesbesselsingular}  converges absolutely and uniformly on $[0,1]$, while the series of the first derivative converges in the $L^1(0,1)$ norm, and the series of the second derivative in $L^1(\delta,1)$ for all $0<\delta<1$. The function $y_0\in  AC[0,1]$ with $y_0'\in AC[\delta,1]$ for all $0<\delta<1$ and satisfies Eq. \eqref{eq:besselsingular} a.e. in $(0,1)$ together with the asymptotic \eqref{eq:asymptsingular}.
	\end{proposition}
	
	\begin{proof}
		The proof follows the same approach as in \cite[Th. 47]{mineradial1}. Using the inequality $|\mathcal{L}_{-\frac{1}{2}}(r,s)|\leq r\sqrt{\frac{r}{s}}$ for $0<s<r\leq 1$ \cite[Lemma 45]{mineradial1}, an induction argument yields the estimates
		\begin{align*}
			|(\psi_k)^{(j)}(r)|\leq \frac{\sqrt{r}r^{k-j}}{k!}\left(\int_0^r|V(r)|dr\right)^k, \quad k\in \mathbb{N}, j=0,1. 
		\end{align*}
		By the Weierstrass M-tests, the series \eqref{eq:seriesbesselsingular} converges absolutely and uniformly $[0,1]$, and the series of the first derivative converges in $L\pot 1(0,1)$. The proof that $y_0$ is a solution of Eq. \eqref{eq:besselsingular} and satisfies the asymptotic \eqref{eq:asymptsingular} is similar to that of Proposition \ref{Prop:Solbessel}.
	\end{proof}	
	
	In the case $d=2$, the generating formal polynomials take the form
	\begin{equation}\label{eq:planarpol}
		\mathcal{V}_m(r,\theta)=r\pot{|m|}\alpha_{|m|}(r)e\pot{im\theta},\quad m\in \mathbb{Z},\; 0\menor r\menor 1,\; 0\leq \theta\leq 2\pi.
	\end{equation}
	
	Applying the same procedure as in Theorem \ref{Th:Sol1}, and using Lemma \ref{Lemma:boundsalpha} (i), we have that $\{\mathcal{V}_m\}_{m\in \mathbb{Z}\setminus\{0\}}$ is orthogonal in $W\pot{1,2}(\mathbb{B}\pot 2)$ (because $r\pot{2m+d-3}=r\pot{2m-1}$ is integrable for $m\in \mathbb{N}$). The critical case occurs at $m=0$, because asymptotic \eqref{eq:asymptsingular} implies that $\alpha_0(r)=o\left(\frac{1}{r}\right)$ as $r\rightarrow 0\pot +$. Hence, we cannot guarantee that $\nabla \mathcal{V}_0$ belongs to $L\pot 2(\mathbb{B}\pot 2)$. However, since $\alpha_0(0)=1$ and $\alpha_0$ is continuous, it follows that $\mathcal{V}_0\in L\pot 2(\mathbb{B}\pot 2)$. In conclusion.
	\begin{proposition}
		For $d=2$, and $V\in L_1^1(0,1)$ satisfying \eqref{eq:hyphotesesTeschl}, the system $\{\mathcal{V}_m\}_{m\in \mathbb{Z}}\subset L\pot 2(\mathbb{B}\pot 2)\cap \operatorname{Sol}_V\pot {dist}(\mathbb{B}\pot 2)$ and is orthogonal in $L\pot 2(\mathbb{B}\pot 2)$.
	\end{proposition}

	In this case, the Poisson kernel has the form
	\begin{equation}\label{eq:poissonplanar}
		P_V(r,\theta,\vartheta)=\sum_{m\in \mathbb{Z}}\frac{r\pot{|m|}\alpha_{|m|}(r)}{\alpha_{|m|}(1)}e\pot{im(\theta-\vartheta)}.
	\end{equation}
	
	We recall that $\varphi\in L\pot{2}(\mathbb{S}\pot 1)$ can be identified with a periodic function $\varphi\in L\pot 2(0,2\pi)$, and 
	\begin{equation}\label{eq:1dfourierseries}
		\varphi(\theta)=\sum_{m\in \mathbb{Z}}\widehat{\varphi}_me\pot{im\theta},\quad \widehat{\varphi}_m=\frac{1}{2\pi}\int_0\pot{2\pi}\varphi(\theta)e\pot{-im\theta}d\theta,\; m\in \mathbb{Z}.
	\end{equation}
	
	\begin{lemma}\label{Lemma:fourierseriesinfinity}
		If $\varphi\in \mathscr{D}(\mathbb{S}^1)$, then $\frac{\partial ^ju_{\varphi}}{\partial \theta^j}\in C(\overline{\mathbb{B}^2})$ for all $j\in \mathbb{N}_0$, and $\varphi_r\rightarrow \varphi$ in $\mathscr{D}(\mathbb{S}^1)$ as $r\rightarrow 1^-$.
	\end{lemma}
	\begin{proof}
		Given $j\in \mathbb{N}_0$, integration by parts yields
		\[
		\widehat{\varphi}_m=\frac{\widehat{\varphi^{(j)}}_m}{(im)^j} \;\;\text{ and then }\;\; \sum_{m\in \mathbb{Z}}|(im)^j\widehat{\varphi}_me^{im\theta}|=\sum_{m\in \mathbb{Z}}|\widehat{\varphi^{(j)}}_m|
		\]
		Furthermore, according to \cite[Ch. 1, Th. 6.2]{katznelson}, the series of the Fourier coefficients of $\varphi^{(j)}$ converges absolutely. By the Weierstrass M test, the series of $\varphi^{(j)}$ converges absolutely and uniformly on $\mathbb{S}^1$, that is, the Fourier series of $\varphi$ converges in $\mathscr{D}(\mathbb{S}^1)$. Now, the series of the partial derivatives in $\theta$ of $u_{\varphi}$ satisfies the estimates
		\[
		\sum_{m\in \mathbb{Z}}\left|\frac{r^m\alpha_m(r)\widehat{\varphi}_m (im)^je^{im\theta}}{\alpha_m(1)}\right|\leq C_0 C_4\sum_{m\in \mathbb{Z}}|\widehat{\varphi^{(j)}}_m|<\infty.
		\]
		
		Consequently, the series of $\frac{\partial^ju_{\varphi}}{\partial \theta^j}$ converges absolutely and uniformly on $\overline{\mathbb{B}^2}$, and $\frac{\partial^ju_{\varphi}}{\partial \theta^j}\in C(\overline{\mathbb{B}^2})$. By the uniform continuity, $\varphi_r^{(j)}=\frac{\partial^ju_{\varphi}(r,\cdot)}{\partial \theta^j}\rightarrow \frac{\partial^ju_{\varphi}}{\partial \theta^j}\big{|}_{r=1}=\varphi^{(j)}$ uniformly on $\mathbb{S}^1$ as $r\rightarrow 1^{-1}$ for all $j\in \mathbb{N}_0$, that is, $\varphi_r\rightarrow \varphi$ in $\mathscr{D}(\mathbb{S}^1)$.
	\end{proof}
	
	If $f\in \mathscr{D}'(\mathbb{S}^1)$, then $f=\sum_{m\in \mathbb{Z}}\widehat{f}_m$, where $\widehat{f}_m:= (f|e\pot{-im\theta})_{\mathscr{D}(\mathbb{S}\pot 1)}, m\in \mathbb{Z}$, and the series converges in $\mathscr{D}'(\mathbb{S}\pot 1)$ with the estimates $|\widehat{f}_m|\leq C|m|\pot{\beta}$, for some $C\mayor 0$ and $\beta\in \mathbb{R}$ \cite[Lemma 2]{estrada}.
	\begin{theorem}
		Given $f\in \mathscr{D}'(\mathbb{S}\pot 1)$, the function 
		\begin{equation}\label{eq:seriesf2d}
			u_f(r,\theta):=\sum_{m\in \mathbb{Z}}\frac{r^m\alpha_{|m|}(r)}{\alpha_{|m|}(1)}\widehat{f}_me^{im\theta},
		\end{equation}
		belongs to $\operatorname{Sol}_V\pot {dist}(\mathbb{B}^2)$ with $(u_f)_{db}=f$.
	\end{theorem}
	\begin{proof}
		
		By repeating the same procedures as in Theorem \ref{Th:Sol2dist}, the series \eqref{eq:seriesf2d} converges uniformly for $|r|\leq R\menor 1$ and $0\leq \theta\leq 2\pi$. Again, using the same argument as in Theorem \ref{Th:Sol2dist}, we have
		\begin{equation}\label{eq:auxfourier1}
			\int_0^{2\pi}f_r(\theta)\varphi(\theta)d\theta=(f|\varphi_r)_{\mathscr{D}(\mathbb{S}\pot 1)}.
		\end{equation}
		By Lemma \ref{Lemma:fourierseriesinfinity}, $\varphi_r\rightarrow\varphi$ in $\mathscr{D}'(\mathbb{S}^1)$ as $r\rightarrow 1^{-}$. Consequently,  the right-hand side of \eqref{eq:auxfourier1} tends to $(f|\varphi)_{\mathscr{D}(\mathbb{S}^1)}$,  which implies that $(u_f)_{db}=f$, as desired. 
	\end{proof}

	\section{Conclusions}
	We present an explicit construction of an orthogonal system of solutions for the radial Schrödinger equation with a radial, complex-valued, integrable potential. This system arises from an explicit construction of solutions to a family of perturbed Bessel equations. An algorithm is provided to facilitate the practical computation of the system.
	
	We establish the conditions under which this system forms an orthogonal basis for the space of weak solutions. Additionally, we construct solutions to the Dirichlet problem with trace boundary conditions and identify sufficient and necessary conditions on the potential that guarantee the uniqueness of the solution.
	
	Finally, we determine the conditions for the existence of a generalized Poisson kernel and solve the Dirichlet problem with distributional boundary conditions.
	\newline

	{\bf Acknowlegdments:}
	The author thanks to Instituto de Matemáticas de la U.N.A.M. Unidad Querétaro (México), where this
	work was developed, and the SECIHTI for their support through the program {\it Estancias Posdoctorales por
		México Convocatoria 2023 (I)}.
	\newline
	
	{\bf Conflict of interest:} This work does not have any conflict of interest.

	\appendix
	
	\section{Appendix}\label{Sec:Appendix}
	
	\begin{proof}[Proof of Lemma \ref{Lemma:sphericalgradient} ]

		Let us consider the spherical coordinates $x=r\xi$, where 
		\begin{equation}\label{eq:apsphericalcoord}
			\xi_i=\begin{cases}
				\prod_{j=1}\pot{d-1}\sin\theta_{j},\quad \text{ if  }\; i=1,\\
				\cos\theta_{i-1}\prod_{j=i}\pot{d-1}\sin\theta_j,\quad \text{ if   }\; 2\leq j\leq d,
			\end{cases}
		\end{equation}
		with $0\menor \theta_1\menor \pi$, $0\menor \theta_j\menor 2\pi$, $j=2,\dots,d-1$. Here, we use the convention that $\prod_{\emptyset}=1$. Let $\Phi:(0,\infty)\times (0,\pi)\times (0,2\pi)\pot{d-1}\rightarrow \mathbb{R}\pot d\setminus\{0\}$ be the spherical changes of variables given by $\Phi(r,\theta_1,\dots,\theta_{d-1})=r\xi$. We identify the radial vector $\mathbf{r}$ with the corresponding point on the unit sphere $\xi\in \Sd$. Hence, the Jacobian matrix of the transformation $\Phi$ is given by $\Phi'=\left[\mathbf{r},\frac{\partial \Phi}{\partial \theta_i},\dots,\frac{\partial \Phi}{\partial \theta_{d-1}}\right]$. Note that $\mathbf{r}$ is the outward normal vector to $\mathbb{S}\pot{d-1}$ at the point $\xi$, and that the tangent space $\mathcal{T}_{\xi}(\mathbb{S}\pot{d-1})$ is spanned by $\left[\frac{\partial \xi}{\partial\theta_i}\right]_{i=1}\pot{d-1}$. Since $\frac{\partial\Phi}{\partial \theta_j}=r\frac{\partial \xi}{\partial \theta_j}$, it follows that  $\mathbf{r}\perp \frac{\partial \Phi}{\partial\theta_i}$, $j=1,\dots, d-1$. The derivatives of \eqref{eq:apsphericalcoord} are
		\begin{equation}
			\left(\frac{\partial \Phi}{\partial \theta_i}\right)_j:=r\cdot 
			\begin{cases}
				\cos\theta_i\prod_{\overset{k=1
					}{k\neq i}}\pot{d-1}\sin\theta_k, & \text{if  }j=1,\\
				\cos\theta_i\cos\theta_{j-1}\prod_{\overset{k=j
					}{k\neq i}}\pot{d-1}\sin\theta_k, & \text{if  }1\menor j\leq i,\\
				-\prod_{k=i}\pot{d-1}\sin\theta_k, & \text{ if  } j=i+1,\\
				0, & \text{ otherwise},
			\end{cases}\quad i=1,\dots, d-1; j=1,\dots, d.
		\end{equation}
		To simplify the notation, we set $c_j:=\cos\theta_j$ and $s_j:=\sin\theta_j$. If $i\menor s$, then
		\begin{align*}
			\frac{\partial\Phi}{\partial\theta_i}\cdot \frac{\partial\Phi}{\partial\theta_s}&=r\pot 2\left\{\prod_{\overset{k=1}{k\neq i}}\pot{d-1}s_kc_i\prod_{\overset{k=1}{k\neq s}}\pot{d-1}s_kc_s+\sum_{j=2}\pot{i}\prod_{\overset{k=j}{k\neq i}}\pot{d-1}s_kc_ic_{j-1}\prod_{\overset{k=1}{k\neq s}}\pot{d-1}s_kc_sc_{j-1}-\prod_{k=i}\pot{d-1}s_k\prod_{\overset{k=i+1}{k\neq s}}\pot{d-1}s_kc_sc_i\right\}\\
			&= r\pot 2 c_ic_s\left\{s_is_s\prod_{\overset{k=1}{k\neq i,s}}\pot{d-1} s_k\pot 2+\sum_{j=2}\pot{i}c_{j-1}\pot 2 s_is_s\prod_{\overset{k=j}{k\neq i,s}}\pot{d-1} s_k\pot 2-s_is_s\prod_{\overset{k=i+1}{k\neq s}}\pot{d-1} s_k\pot 2\right\}\\
			&= r\pot 2 c_ic_ss_is_s\left\{ \prod_{\overset{k=2}{k\neq i,s}}\pot{d-1}s_k\pot 2+\sum_{k=3}\pot i \prod_{\overset{k=j}{k\neq i,s}}\pot{d-1}c_{j-1}\pot 2s_k\pot 2-\prod_{\overset{k=i+1}{k\neq s}}\pot{d-1} s_k\pot 2\right\}\\
			\vdots &\\
			&= r\pot 2 c_ic_ss_is_s\left\{ \prod_{\overset{k=i+1}{k\neq s}}\pot{d-1}s_k\pot 2-\prod_{\overset{k=i+1}{k\neq s}}\pot{d-1}s_k\pot 2\right\}=0.
		\end{align*}
		In the same way,
		\begin{align*}
			\left|\frac{\partial\Phi}{\partial \theta_i}\right|\pot 2 &= r\pot 2 \left\{\prod_{\overset{k=1}{k\neq i}}^{d-1}s_k\pot 2c_i\pot 2+\sum_{j=2}\pot i \prod_{\overset{k=j}{k\neq i}}\pot{d-1}s_k\pot 2c_i\pot 2c_{j-1}\pot 2+\prod_{k=i}\pot{d-1}s_k\pot 2 \right\}\\
			&= r\pot 2\left\{ c_i\pot 2\left[\prod_{\overset{k=1}{k\neq i}}\pot{d-1}s_k\pot 2+\sum_{j=2}\pot i \prod_{\overset{k=j}{k\neq i}}\pot{d-1}s_k\pot 2c_{j-1}\pot 2\right] +\prod_{k=i}\pot{d-1}s_k\pot 2\right\}\\
			&= r\pot 2\left\{ c_i\pot 2\left[\prod_{\overset{k=2}{k\neq i}}\pot{d-1}s_k\pot 2+\sum_{j=3}\pot i \prod_{\overset{k=j}{k\neq i}}\pot{d-1}s_k\pot 2c_{j-1}\pot 2\right] +\prod_{k=i}\pot{d-1}s_k\pot 2\right\}\\
			\vdots &\\
			&= r\pot 2\left\{c_i\pot 2\prod_{k=i+1}\pot{d-1}s_k\pot 2+s_i\pot 2\prod_{k=i+1}\pot{d-1}s_k\pot 2\right\}=r\pot 2\prod_{k=i+1}\pot{d-1}s_k\pot 2.
		\end{align*}
		Since $0\menor \theta_j\menor \pi$ for $j=2,\dots,d-1$, it follows that $\left|\frac{\partial\Phi}{\partial \theta_j}\right|=r\prod_{k=j+1}\pot{d-1}\sin\theta_k$. Let $\Theta_j:=\left|\frac{\partial \xi}{\partial \theta_j}\right|$. Thus, $\Theta_j=\prod_{k=j+1}\pot{d-1}\sin\theta_k$. Since the columns of the matrix $\Phi'$ are orthogonal, we have $(\Phi')\pot T\Phi'=\Phi'(\Phi')\pot T=\operatorname{diag}(1,r^2\Theta_1^2,\dots, r^2\Theta_{d-1}^2)$. Thus, $(\Phi')\pot{-1}=\left[\mathbf{r},\frac{1}{r\pot 2\Theta_1\pot 2}\frac{\partial \Phi}{\partial \theta_1},\dots,\frac{1}{r\pot 2\Theta_{d-1}\pot 2}\frac{\partial \Phi}{\partial \theta_{d-1}}\right]\pot T$.
		By the chain rule $\nabla_{r,\theta_1,\dots,\theta_{d-1}} u(r\xi)=\nabla_x u(x)\Phi'$, and consequently, 
		\begin{align*}
			\nabla_xu(x)&= \nabla_{r,\theta_1,\dots,\theta_{d-1}} u(r\xi)(\Phi')\pot{-1}\\ &=\left(\frac{\partial u}{\partial r},\frac{\partial u}{\partial \theta_1},\dots,\frac{\partial u}{\partial\theta_{d-1}}\right)\left[\mathbf{r},\frac{1}{r\pot 2\Theta_1\pot 2}\frac{\partial \Phi}{\partial \theta_1},\dots,\frac{1}{r\pot 2\Theta_{d-1}\pot 2}\frac{\partial \Phi}{\partial \theta_{d-1}}\right]\pot T\\
			&= \frac{\partial u}{\partial r}\mathbf{r}+\frac{1}{r}\sum_{j=1}\pot{d-1}\frac{1}{\Theta_j}\frac{\partial u}{\partial \theta_j}\widehat{\boldsymbol{\theta}}_j,
		\end{align*}
		where $\widehat{\boldsymbol{\theta}}_j=\frac{1}{\Theta_j}\frac{\partial\xi_j}{\partial\theta_j}$, $j=1,\dots,d-1$.
	\end{proof}
	\newline 
	
	\begin{proof}[Proof of Lemma \ref{Prop:sphericalprop1}] 
		
		Properties (i), (iii), and (iv) can be found in \cite[Ch. V]{axler}, and property (ii) in \cite[Th. 1. 45]{day}. It remains only to prove (v) and (vi).
		\begin{itemize}
			\item[(v)] Let $p\in \mathcal{H}_m(\Sd)$ and $q\in \mathcal{H}_m(\Sd)$, with $m\neq n$, and consider the corresponding homogeneous harmonic polynomials $P(x)=r\pot mp(\xi)$, $Q(x)=r\pot m q(\xi)$. Since $P$ and $Q$ are harmonic, the first Green identity yields
			\[
			\int_{\Bd}\nabla P\cdot \overline{\nabla Q}= \int_{\Sd}P\overline{\frac{\partial Q}{\partial \nu}}d\sigma= m\int_{\Sd}p\overline{q}d\sigma=0.
			\]
			
			Here, we use \eqref{eq:normalder} to obtain $\frac{\partial Q}{\partial \nu}=\frac{\partial r\pot m q}{\partial r}\bigg{|}_{r=1}=mq$. Now, 
			\begin{align*}
				\nabla P\cdot \overline{\nabla Q}= \frac{\partial P}{\partial r}\frac{\partial Q}{\partial r}+\nabla_{\Sd}P\cdot \overline{\nabla_{\Sd}Q}
				=mnr\pot{m+n-2}pq+r\pot{m+n}\nabla_{\Sd}p\cdot \overline{\nabla_{\Sd}q}.
			\end{align*}
			Consequently,
			\begin{align*}
				0=&\int_{\Bd}\nabla P\cdot \overline{\nabla Q}=mn\int_0^1r^{d+m+n-3}dr\int_{\Sd}pqd\sigma  +\int_0^1r^{d-2+m+n}dr\int_{\Sd}\nabla_{\Sd}p\cdot \overline{\nabla_{\Sd}q}d\sigma\\
				=& \frac{1}{d-1+m+n}\int_{\Sd}\nabla_{\Sd}p\cdot \overline{\nabla_{\Sd}q}d\sigma
			\end{align*}
			(the integral over $(0,1)$ in the second line is well-defined because $d-1\geq 0$). Therefore, $\nabla_{\Sd}p\perp_{L\pot 2(\Sd)} \nabla_{\Sd}q$.
			\item[(vi)] Set $P=r\pot m Y_j\pot{(m)}$ and $Q=r\pot m Y_k\pot{(m)}$. First, suppose that $j\neq k$. Applying the same procedure as in the previous point, we obtain
			\begin{align*}
				\nabla P\cdot \overline{\nabla Q}
				&=m\pot 2r\pot{2m-2}Y_j\pot{(m)}Y_k\pot{(m)}+r\pot{2m}\nabla_{\Sd}Y_j\pot{(m)}\cdot \overline{\nabla_{\Sd}Y_k\pot{(m)}}
			\end{align*}
			By \eqref{eq:relatriongradients1},
			\[
			\int_{\Sd} \nabla P\cdot \overline{\nabla Q}d\sigma  =m(d+2m-2)\int_{\Sd}P\overline{Q}d\sigma=0,
			\]
			and also
			\begin{align*}
				\int_{\Sd} \nabla P\cdot \overline{\nabla Q}d\sigma= m\pot 2\int_{\Sd} Y_j\pot{(m)}\overline{Y_k\pot{(m)}}d\sigma+\int_{\Sd} \nabla_{\Sd} Y_j\pot{(m)}\cdot \overline{\nabla Y_k\pot{(m)}}d\sigma.
			\end{align*}
			Since $\int_{\Sd} Y_j\pot{(m)}\overline{Y_k\pot{(m)}}d\sigma=0$, we conclude \eqref{eq:ortogonalgrad1}.
			Finally, for the norm of $\nabla_{\Sd}Y_j\pot{(m)}$, using \eqref{eq:relatriongradients1} we get 
			\begin{align*}
				\int_{\Sd}| Y_j\pot{(m)}|\pot 2d\sigma=\frac{1}{m(d+2m-2)}\left(m\pot 2\int_{\Sd}|Y_j\pot{(m)}|\pot 2d\sigma +\int_{\Sd}|\nabla_{\Sd}Y_j\pot{(m)}|\pot 2d\sigma\right),
			\end{align*}
			from where we obtain
			\begin{align*}
				\frac{1}{m(d+2m-2)}\int_{\Sd}|\nabla_{\Sd}Y_j\pot{(m)}|\pot 2d\sigma=1-\frac{m\pot 2}{m(d+2m-2)}=\frac{m(m+d-2)}{m(d+2m-2)},
			\end{align*}
			and we conclude \eqref{eq:normgradspherical}.
		\end{itemize}
		
	\end{proof}

	
\end{document}